\documentclass[12pt]{article}

\usepackage[colorlinks=true, pdfstartview=FitV, linkcolor=blue, citecolor=blue, urlcolor=blue]{hyperref}

\usepackage{amssymb,amsmath, amscd,mathdots}
\usepackage{times, verbatim}
\usepackage{graphicx}

\DeclareFontFamily{OT1}{rsfs}{}
\DeclareFontShape{OT1}{rsfs}{n}{it}{<-> rsfs10}{}
\DeclareMathAlphabet{\mathscr}{OT1}{rsfs}{n}{it}

\newtheorem{theorem}{Theorem}
\newtheorem{corollary}[theorem]{Corollary}
\newtheorem{lemma}[theorem]{Lemma}
\newtheorem{proposition}[theorem]{Proposition}
\newtheorem{remark}[theorem]{Remark}
\newenvironment{proof}{\noindent {\bf Proof:}}{$\Box$ \vspace{2 ex}}

\DeclareMathOperator{\Aut}{Aut}

\DeclareMathOperator{\SL}{SL}
\DeclareMathOperator{\GL}{GL}

\DeclareMathOperator{\Sel}{Sel}

\def\Z{{\mathbb Z}}

\def\sf{{\rm Sel}^{(2)}_{\rm fake}(C/\Q)}
\def\SL{{\rm SL}}
\def\GL{{\rm GL}}

\def\Stab{{\rm Stab}}
\def\Sym{{\rm Sym}}
\def\Jac{{\rm Jac}}

\def\P{{\mathbb P}}
\def\Disc{{\rm Disc}}

\def\ss{{\rm ss}}

\def\Aut{{\rm Aut}}

\def\red{{\rm red}}
\def\Det{{\rm Det}}
\def\Vol{{\rm Vol}}
\def\R{{\mathbb R}}
\def\F{{\mathbb F}}
\def\FF{{\mathcal F}}
\def\RR{{\mathcal R}}

\def\Q{{\mathbb Q}}

\def\H{{\mathcal H}}

\def\J{{\mathcal J}}
\def\C{{\mathcal C}}

\def\Z{{\mathbb Z}}
\def\P{{\mathbb P}}
\def\F{{\mathbb F}}
\def\Q{{\mathbb Q}}
\def\C{{\mathbb C}}

\def\H{{\mathcal H}}

\def\BB{{\mathcal B}}

\title{Most hyperelliptic curves over $\Q$ have no rational points\footnote{Or, as suggested by Don Zagier: ``Most hyperelliptic curves are pointless".}}

\author{Manjul Bhargava}

\begin{document}

\maketitle


\section{Introduction}

The purpose of this article is to show that most hyperelliptic curves
over $\Q$ have no rational points.  By a hyperelliptic curve over
$\Q$, we mean here a smooth, geometrically irreducible, complete curve $C$ over $\Q$ equipped with a fixed
map of degree 2 to $\P^1$ defined over $\Q$.
Thus any hyperelliptic curve $C$ over $\Q$ of genus~$g$ can be 
embedded in weighted projective space $\P(1,1,g+1)$ and 
expressed by an equation of the form
\begin{equation}\label{hypereq}
C: z^2 = f(x,y) = f_0x^n+f_1x^{n-1}y+\cdots+f_ny^n
\end{equation}
where $n=2g+2$, the coefficients $f_i$ lie in $\Z$, and $f$ factors into distinct linear factors over $\bar\Q$.  
Define the height $H(C)$ of~$C$~by
$$ H(C):= H(f):=\max\{|f_i|\}.$$
Then we prove:

\begin{theorem}\label{main}
  As $g\to\infty$, a density approaching $100\%$ of hyperelliptic
  curves over $\Q$ of genus $g$ possess no rational points.
\end{theorem}
More precisely, let $\rho_g$ denote the lower density of
hyperelliptic curves over $\Q$ of genus $g$, when ordered by height,
that have no rational points.  \,Then \,$\rho_g\to1$ as $g\to\infty$.
In fact, we prove that 
$\rho_g=1-o(2^{-g})$, 
so the convergence to 1 is quite rapid.
Theorem~1 may thus be viewed as a ``strong asymptotic form of Faltings' Theorem'' for hyperelliptic~curves over~$\Q$.

Since, by the work of Poonen and Stoll~\cite{PS}, for each $g\geq 1$ a density of at least 
$80\%$ of hyperelliptic curves of over $\Q$ of genus $g$ are {\it locally soluble}---i.e., have a point locally at every place of $\Q$---we obtain:

\begin{corollary}\label{hasse}
As $g\to\infty$, a density approaching $100\%$ of locally soluble hyperelliptic curves over $\Q$ of genus $g$ fail the 
Hasse principle.
\end{corollary}
Indeed, we obtain the same estimate of $1-o(2^{-g})$ on the lower density
$\rho_g'$ of locally soluble hyperelliptic curves over~$\Q$ of genus $g$ that fail the
Hasse principle. 

Our strategy to prove Theorem~1 
is to relate rational points on hyperelliptic curves to elements of the representation 
$\Z^2\otimes\Sym_2\Z^n$ of $\GL_n(\Z)$.
We view an element $v$ of $\Z^2\otimes\Sym_2\Z^n$ as a pair $(A,B)$ of
symmetric $n\times n$ matrices with entries in~$\Z$.  Then an element
$g\in \GL_n (\Z)$ acts on $(A,B)$ by the formula $g\cdot(A, B) =
(gAg^t , gBg^t)$.  To such a pair $v=(A,B)$, we may associate a binary
form $f_v$ of degree $n$, given by
\begin{equation}\label{invform} 
f_v(x,y) = (-1)^{n/2}\det(Ax-By).
\end{equation}
The coefficients of $f_v$ in fact generate the ring of polynomial
invariants for the action of $\GL_n(\Z)$ 
on $\Z^2\otimes\Sym_2\Z^n$
(see, e.g., the work of Schwarz~\cite{Schwarz}).  We thus call $f_v$
the {\it invariant binary $n$-ic form} associated to $v\in
\Z^2\otimes\Sym_2\Z^n$.

The orbits
of $\GL_n(\Z)$ on $\Z^2\otimes\Sym_2\Z^n$ were first considered in the
case $n=2$ by Hardy and Williams~\cite{HW} and more generally by
Morales~\cite{Morales1,Morales}.  A classification of the orbits in the cases
$n=2$ and $n=3$, in terms of ideal classes in quadratic and cubic
rings, was given in \cite{hclI} and \cite{hclII}, while a complete
classification for general $n$ in terms of module classes of
rings of rank $n$ was given by Wood~\cite{Wood1}.

The key algebraic construction that allows us to prove
Theorem~\ref{main} and Corollary~\ref{hasse}
is the utilization of a rational
point on a hyperelliptic curve (\ref{hypereq}) to produce an element
$v\in \Z^2\otimes\Sym_2\Z^n$ such that $f_v=f$. That is, a rational
point on the hyperelliptic curve $z^2=f(x,y)$ allows us to construct an
integral orbit for the action of $\GL_n(\Z)$ on
$\Z^2\otimes\Sym_2\Z^n$ whose invariant binary $n$-ic form is equal to~$f$!  On the other hand, we show by geometry-of-numbers arguments
that, for most integral binary $n$-ic forms $f$, there do not exist
any such integral orbits with invariant binary form equal to
$f$. These two facts together will then yield Theorem~\ref{main} (and
hence also Corollary~\ref{hasse}).

The construction in fact allows us to study not just rational points on the curves $C$, but also elements of what is called the ``fake 2-Selmer set'' of $C$, in the terminology of Bruin and Stoll~\cite{BruinStoll}; this set also has its origins in the works of Cassels~\cite{Cassels2} and Poonen and Schaefer~\cite{PSc}.   The {\it fake 2-Selmer set} $\sf$ of $C:z^2=f(x,y)$ is a certain finite subset of $K^\times/(K^{\times2}\Q^\times)$, where $K$ is the $\Q$-algebra defined by $\Q[x]/(f(x,1))$ if $f_0\neq 0$ (see~\cite[\S2]{BruinStoll} or \S\ref{mainproof} for more precise definitions).  Geometrically, the fake 2-Selmer set corresponds 
to a set of equivalence classes of two-covers of $C$,
which have points everywhere locally, such that each rational point of $C$ lifts to a rational point on at least one of these two-covers~\cite[\S3]{BruinStoll}.  In particular, the emptiness of $\sf$ guarantees the nonexistence of points on $C$. 

To get a handle also on the fake 2-Selmer set, we show that not just rational points on $C:z^2=f(x,y)$ but also general elements of the fake 2-Selmer set can be used to construct an orbit of $\GL_n(\Z)$ on $\Z^2\otimes\Sym_2\Z^n$ having invariant form $f$.  Applying the above-described counting methods, we are then able to determine an upper bound on the average size of the fake 2-Selmer set of hyperelliptic curves over $\Q$ of genus $g$; this upper bound decays exponentially in $g$.  Specifically, we prove:

\begin{theorem}\label{fake}
The $($limsup of the$)$ average size of the fake 2-Selmer set of hyperelliptic curves over $\Q$ of genus $g$ is $o(2^{-g})$.
\end{theorem}
This result thus allows us to pinpoint the obstruction to having rational points on the curves occurring in Theorem~\ref{main}.  Indeed, the geometric interpretation of the fake 2-Selmer set of $C$ means that if $\sf$ is empty, then there are no locally soluble 2-coverings of $C$.
As shown by Skorobogatov~\cite{Sk}, using the descent theory of Colliot-Th\'el\`ene and Sansuc~\cite{CS}, such a ``2-covering obstruction'' yields
a case of the Brauer--Manin obstruction.  We therefore obtain: 

\begin{corollary}\label{brauer}
As $g\to\infty$, a density approaching $100\%$ of hyperelliptic curves over $\Q$ of genus $g$ 
have empty Brauer set, i.e., have a 
Brauer--Manin obstruction to having a rational point.
\end{corollary}

Our method of proof also 
enables us to prove 
more precise statements for individual genera in Theorem~\ref{main}, Corollary~\ref{hasse}, Theorem~\ref{fake}, and Corollary~\ref{brauer}. Indeed, we prove that, for each $g\geq 1$, a~positive~proportion of locally soluble hyperelliptic curves over $\Q$ of genus $g$ have no rational points and thus fail the Hasse principle;
moreover, this failure is explained by the Brauer--Manin obstruction. 
The proportion of locally soluble hyperelliptic curves failing the Hasse principle in this way is shown to be greater than 10\% for genus~1 curves, 
and greater than 50\% for genus 2 curves;
meanwhile, once the genus is at least 10, the corresponding proportion exceeds~99\%.

We note that analogues of Theorem~1 have also been proven recently for certain special families of hyperelliptic curves, namely, those with a marked rational Weierstrass point (such curves have models of the form $z^2=f(x)$ where $f$ is an odd degree polynomial)  and those with a marked rational non-Weierstrass point (these curves have models of the form $z^2=f(x)$ where $f$ is a monic even degree polynomial).  In joint work with Gross~\cite{BG2}, we proved that the average size of the 2-Selmer group of the family of hyperelliptic curves with a marked rational Weierstrass point is 3.  By developing a beautiful refinement of the method of Chabauty, Coleman, and McCallum~\cite{mc}, Poonen and Stoll~\cite{PS2} use this result on Selmer groups to show that a positive proportion of odd degree hyperelliptic curves $z^2=f(x)$ of genus $g\geq 3$ have only one rational point, namely, the Weierstrass point at infinity; moreover, this proportion tends to 1 as $g\to\infty$. 
The corresponding positive proportion statement for $g=1$ was proven in \cite{BS2}, while the case $g=2$ remains open.

 In~\cite{SW}, Shankar and Wang develop analogues of these methods to show that the average size of the 2-Selmer group for the family of hyperelliptic curves with a marked rational non-Weierstrass point is 6, and that a positive proportion of monic even degree hyperelliptic curves $z^2=f(x)$ of genus $g\geq 10$ have only two rational points, namely the two marked rational non-Weierstrass points at infinity; again, the proportion tends to 1 as $g\to\infty$.

The methods of this paper fit into a larger context.  As in \cite{BS,BS2,BH,BG2,SW}, the problem at hand involves, in particular, counting the number of stable integral orbits of bounded height in a suitable coregular representation.  
A {\it coregular representation} is a representation for which the ring of invariants is a free polynomial ring.  As we have mentioned, the representation  $\C^2\otimes\Sym_2\C^n$ of $\SL_n(\C)$ is coregular as well, but it differs from the representations occurring in prior works in an important way, namely, it does not possess a {\it Kostant section} (in particular, it does not arise in Vinberg's theory of {\it theta groups}).  In other words, the map from our representation to the tuple of generating invariants does not have a section defined over $\Z$.  As a result, sometimes certain values of the invariants do not occur in this representation over $\Z$ (or over~$\Q$).  This is in essence what allows us to show that many curves possess {\it no} rational points.  

On the other hand, the study of such a representation without a Kostant section presents a number of new and interesting challenges, both algebraically and analytically.  For example, algebraically, the orbits with given invariants (even after imposing local conditions) do not form a group in any natural way, unlike all previous cases mentioned above.  Indeed, the relevant set of orbits having given invariants may be empty, {or} it may naturally form a torsor for a group; see \cite{AITII} for further discussion of the algebraic structure of orbits in representations of this type.  Analytically, as far as counting orbits, there are again a number of new features and difficulties that must be overcome in the geometry-of-numbers arguments to handle the cuspidal regions of the relevant fundamental domains; the main differences are caused by the fact that there is no one primary cuspidal region containing the Kostant section, which was a key feature in previous cases.  This paper represents the first time that the stable integral points of bounded height have been counted in a coregular representation having no Kostant section.  We suspect that these new counting, and related sieve, techniques may be useful in the context of other such representations,  which in turn may be useful for proving the nonexistence of other arithmetic objects.

The space $\Z^2\otimes\Sym_2\Z^n$, as we have mentioned, may be viewed as the space of pairs of $n$-ary quadratic forms $(A,B)$, and any such pair $(A,B)$ spans a pencil of quadrics in $\P^{n-1}$.  The geometric relation between pencils of quadrics and hyperelliptic curves, at least over $\C$, was studied by Reid~\cite{R}, Desale and Ramanan~\cite{R}, and Donagi~\cite{D}, while 
the scenario
 over a general field
 was studied 
 by Wang~\cite{W}.  From this point of view, it is indeed very natural to consider the representation $\Z^2\otimes\Sym_2\Z^n$ over the integers to study the arithmetic of hyperelliptic curves. In the forthcoming work~\cite{AITII}, we take this viewpoint considerably further, and study the relation between the orbits of this representation and hyperelliptic curves, over general fields and over $\Z$, in more detail.  This will allow us to extract much additional arithmetic information of interest about general hyperelliptic curves, such as the average size of the {``fake 2-Selmer group''}, points over extensions, and more.  
We will mention some of these forthcoming results in Section~\ref{mainproof}. 

This paper is organized as follows.  In Section~\ref{rptovz}, we give an explicit
construction of an integral orbit in our representation from a
rational point on a hyperelliptic curve.  The construction has a
natural interpretation in terms of the work of Wood~\cite{Wood} on rings and
ideals associated to integral binary $n$-ic forms, and we also
describe this interpretation.  In Section~\ref{fieldorbits}, we then use Wood's
results to classify orbits of the above representation over a general
base field $k$. In particular, we closely study the case where $k$ is 
the field of real numbers or is a finite field, and we determine those
orbits that can arise from a $k$-rational point on a hyperelliptic curve
via the construction of Section~\ref{rptovz}.

In Section~\ref{counting}, we then count the total number of integral orbits on
$\Z^2\otimes\Sym_2\Z^n$ having an irreducible invariant binary $n$-ic
form of bounded height.  The counting methods used build on those in
\cite{dodpf}, as well as those in our joint work with
Shankar~\cite{BS}, Ho~\cite{BH}, and Gross~\cite{BG}, although as mentioned earlier the current
scenario also has some new features in the cusps that must be dealt with.
With the latter counting result in hand, in Section~\ref{mainproof} we
then show via a sifting argument that most binary $n$-ic
forms $f$ cannot have the type of integral orbit constructed in
Section~\ref{rptovz}, thereby completing the proof of Theorem~\ref{main}.  We 
also use analogous arguments to study the fake 2-Selmer set, 
thus proving Theorem~\ref{fake}. 
Finally, in Section~\ref{examples}, 
we  work out details of some cases of small genus, allowing us to prove 
for each $g\geq 1$ that a positive proportion of locally soluble hyperelliptic curves
of genus $g$ have no rational~points.

\section{Construction of an integral orbit associated to a rational point on a hyperelliptic curve}\label{rptovz}

Let 
\begin{equation}
f(x,y) = f_0x^n+f_1x^{n-1}y+\cdots+f_ny^n
\end{equation}
be an binary $n$-ic form over $\Z$ having nonzero discriminant and
nonzero leading coefficient $f_0$, and let
$K_f:=\Q[x]/(f(x,1))=\Q[\theta]$, where $\theta$ denotes the image of
$x$ in the $\Q$-algebra $K_f$.  Then we may associate to $f$
a $\Z$-module
$R_f$ in $K_f$ having basis $\langle
1,\zeta_1,\zeta_2,\ldots,\zeta_{n-1}\rangle$, where
\begin{equation}\label{zetadef}
\zeta_k=f_0\theta^k+f_1\theta^{k-1}+\cdots+f_{k-1}\theta.
\end{equation}
It is a theorem of Birch and Merriman~\cite{BM} that
$\Disc(f)=\Disc(R_f)$, and furthermore it is a theorem of
Nakagawa~\cite[Prop.\ 1.1]{Nakagawa} that $R_f$ is actually a ring of rank $n$
over $\Z$ and thus an order in $K_f$\pagebreak; multiplication is given by
\begin{equation}\label{multtable}
\zeta_i\zeta_j = \sum_{k=j+1}^{\min(i+j,n)}f_{i+j-k}\zeta_k -  
\sum_{k=\max(i+j-n,1)}^i f_{i+j-k}\zeta_k
\end{equation} 
for $1\leq i\leq j\leq n-1$, where we set $\xi_n=-f_n$.  The
construction can also be carried out canonically for binary $n$-ic
forms having discriminant zero, and even for the zero form; see
Wood~\cite{Wood}.

We may define $\Z$-modules $I_f(k)\subset K_f$, for
$k=0,\ldots, n-1$, by
\begin{equation}
I_f(k) = \langle
1,\theta,\theta^2,\ldots,\theta^k,\zeta_{k+1},\ldots,\zeta_{n-1}
\rangle
\end{equation}
which are in fact $R_f$-modules!  They are all powers of one ideal $I_f=I_f(1)$, namely, 
$I_f(k)=I_f^k$.  We have $I_f(0)=R$, while $I_f(n-2)$
is the ``inverse different'' or the ``dualizing module''.   

Wood shows that $R_f$ is the ring of global functions on the subscheme cut out by
$f$ in $\P^1$,  and $I_f(k)=I_f^k$, as an $R_f$-module, 
consists of the sections of the pullback of $\mathcal O(k)$ from $\P^1$.
Hence if $f$ and $f'$ are equivalent forms, via an
element $\gamma \in \SL_2(\Z)$,
then $R_f$ and $R_{f'}$ are naturally identified
too via the corresponding action of $\gamma$ on $\P^1$.  The ideals $I_f(k)$ and
$I_{f'}(k)$ are {\it not} necessarily the same under this identification, however.
If  we write $f'(x,y)=f((x,y)\cdot\gamma)$, where
$\gamma=\bigl(\begin{smallmatrix}
a & b\\ c& d\end{smallmatrix}\bigr)$, then the root $\theta'$ of
$f'(x,1)$ satisfies $\theta'=\frac{d\theta-c}{-b\theta+a}$, and 
\begin{equation}\label{ifif}
I_{f'}(k) = \left(
{-b\theta+a}\right)^{-k} I_f(k).\end{equation}
The map $\delta\mapsto (-b\theta+a)^{-k}\cdot\delta$ gives an explicit isomorphism between $I_f(k)$ and $I_{f'}(k)$ as
$R_f=R_{f'}$-modules.  

Note that the ring $R_f$ and the ideals $I_f(k)$ come with natural
bases.  We thus call the $I_f(k)$ {\it based ideals} of $R_f$, i.e.,
ideals of $R_f$ of rank $n$ over $\Z$ equipped with an ordered
$\Z$-basis.  We define the {\it norm} $N(I)$ of a based ideal as the
determinant of the $\Z$-linear transformation taking the chosen basis
of $I$ to the basis of $R_f$.  Finally, if $I$ is a based ideal, then
$\kappa I$ is also a based ideal in which each basis element of $I$ is
multiplied by $\kappa$.  Thus $N(\kappa I)=N(\kappa)\cdot N(I)$ for
any based ideal $I$.

We then have the following general theorem of Wood~\cite[Thms.\ 3.1 \&
5.5]{Wood1}, phrased in the above terminology (see~\cite[Thm.\
11]{hclI} and \cite[Thm.\ 2]{hclII} for details in the cases $n=2$ and $n=3$):

\begin{theorem}[Wood]\label{wood2nn}
The elements of
$\Z^2\otimes\Z^n\otimes\Z^n$ having a given invariant binary $n$-ic form $f$ with nonzero discriminant
are in canonical bijection with equivalence classes of triples
$(I,I',\alpha)$, where $I,I'\subset K_f$ are based ideals 
of $R_f$, $\alpha\in K_f$, $II'\subseteq \alpha I_f^{n-3}$ as ideals, and
$N(I)N(I')=N(\alpha)N(I_f^{n-3})$.  $($Here, two triples
$(I,I',\alpha)$ and $(J,J',\beta)$ are {\em equivalent} if 
there exist $\kappa,\kappa'\in K_f^\times$ such that $J=\kappa I$, $J'=\kappa' I'$,
and $\beta=\kappa\kappa'\alpha.)$
\end{theorem}
The way to recover a pair $(A,B)$ of $n\times n$ matrices from a
triple $(I,I',\alpha)$ is by taking the coefficients of $\zeta_{n-1}$
and $\zeta_{n-2}$ in the image of the map
$\frac1\alpha\times:I\times I'\to I_f^{n-3}$ in terms of the $\Z$-bases of $I$ and $I'$. 

The symmetric version of this theorem is as follows 
(see~\cite[Thms.\ 4.1 \& 5.7]{Wood1}, or \cite[Thm.\ 16]{hclI} and
\cite[Thm.\ 4]{hclII}
for the special cases $n=2$ and $n=3$):
\begin{theorem}[Wood]\label{woodsym}
The 
elements of $\Z^2\otimes\Sym_2\Z^n$ having a given invariant binary $n$-ic form $f$ 
with nonzero discriminant
are in canonical bijection with equivalence classes of pairs
$(I,\alpha)$, where $I\subset K_f$ is a based ideal of $R_f$, $\alpha\in
K_f$, $I^2\subseteq \alpha I_f^{n-3}$ as ideals, and
$N(I)^2=N(\alpha)N(I_f^{n-3})$.
$($Here, two pairs
$(I,\alpha)$ and $(J,\beta)$ are {\em equivalent} if 
there exists $\kappa\in K_f^\times$ such that $J=\kappa I$ and $\beta=\kappa^2\alpha.)$
\end{theorem}
Here, $\Sym_2\Z^n$ denotes the $\Z$-module of symmetric $n\times n$ 
matrices having entries in $\Z$.  Again, the way to recover a pair $(A,B)$ of symmetric $n\times n$ matrices from a
triple $(I,\alpha)$ is by taking the coefficients of $\zeta_{n-1}$
and $\zeta_{n-2}$ in the image of the map
\begin{equation}\label{mt1}
\frac1\alpha\times:I\times I\to I_f^{n-3}
\end{equation}
 in terms of the $\Z$-basis of $I$. 

\vspace{.1in}
With these preliminaries in hand, we now turn to the question of when there exists a 
pair $(A,B)\in \Z^2\otimes\Sym_2\Z^n$ of symmetric $n\times n$ 
matrices over $\Z$ having a given binary $n$-ic form $f$ as its invariant
form.  If $n$ is odd, then it always exists: we simply take the
element $(A,B)$ in $\Z^2\otimes\Sym_2\Z^n$ corresponding to the pair
$(I_f^{(n-3)/2},1)$.  Here, we use the fact that $I_f^k\cdot I_f^{k'}=
I_f^{k+k'}$; also, $N(I_f^k) = 1/f_0^k$, implying that
$N(I_f^{(n-3)/2})^2 = N(I_f^{n-3})$.

Now let $n$ be even.  Then the above construction of a square root of $I_f^{n-3}$ does not work.
However, we now show how one can construct a pair of  $n\times
n$ symmetric matrices over $\Z$ having invariant form $f$ using a rational point
on the hyperelliptic curve $z^2=f(x,y).$
 
First, suppose that we have a symmetric orbit 
with invariant form $f$ corresponding to the pair $(I,\alpha)$.  Then $N(I)^2 = N(\alpha)\cdot
N(I_f^{n-3})=N(\alpha)/f_0^{n-3}$.
If $\alpha=a\theta+b$, then this gives a rational point on
$z^2=f(x,y)$, since then $$N(\alpha)=a^nN(\theta+b/a) = \frac{a^nf(-b/a,1)}{f_0}
= \frac{f(-b,a)}{f_0},$$ 
implying that $$f(-b,a)=
{f_0N(\alpha)}=\frac{f_0N(I)^2}{N(I_f^{n-3})}={f_0^{n-2}N(I)^2}
$$ is a square.  Thus $(-b,a,f_0^{(n-2)/2}N(I))$ yields a rational point
on $z^2=f(x,y)$.

Next, suppose that $(x_0,y_0,c)$ is a rational point on $z^2=f(x,y)$.  By scaling the coordinates of this point appropriately, we may assume that $x_0$ and $y_0$ are integral and relatively prime.  Using this point $(x_0,y_0,z_0)$ on $z^2=f(x,y)$, we construct an element of $\Z^2\otimes\Sym_2\Z^n$ having invariant form $f$ as follows.  First, using an $\SL_2(\Z)$-transformation $\gamma$ on $f$, we may assume that $(x_0,y_0)=(0,1)$, so that $f_n=c^2$.  (Indeed, if we produce an element $(A,B)$ in $\Z^2\otimes\Sym_2\Z^n$ yielding this new $f$ as its invariant form, then applying $\gamma^{-1}$ to $(A,B)$ will produce the desired element yielding the original $f$) .
Now set $\alpha=\theta$, and 
note that
\begin{equation}
\theta I_f^{n-3} = \langle
c^2,\theta,\theta^2,\ldots,\theta^{n-2},f_0\theta^{n-1}
\rangle .
\end{equation}
Let \begin{equation}\label{construct}I\,=\,\langle c,\theta I_f^{(n-4)/2}\rangle 
\,=\, \langle c,\theta,\theta^2,\ldots,\theta^{(n-2)/2},\zeta_{n/2},\ldots,\zeta_{n-1}\rangle.
\end{equation}
Then one easily checks using (\ref{zetadef}) that $$I^2\subseteq \theta\cdot I_f^{n-3},$$ and furthermore $$N(I)^2 = [c/f_0^{(n-2)/2}]^2 = 
N(\theta)N(I_f^{n-3}).$$  
Thus the pair $(I,\alpha)$ gives a point in $\Z^2\otimes\Sym_2\Z^n$ with
invariant form $f$, corresponding to our integral point $(0,1,c)$ on $z^2=f(x,y)$.

The explicit pair $(A,B)\in\Z^2\otimes\Sym^2\Z^n$ we have constructed above corresponding to the integral point
$(0,1,c)$ on $z^2=f(x,y)$ is given by
\begin{equation*}
\left(\left[
\begin{array}{cc}
-1 & 0 \\ 0 & f_0 
\end{array}\right],
\left[\begin{array}{cc}
0 & c \\ c & -f_1 
\end{array}\right]\right)
\end{equation*}
for $n=2$; by
\begin{equation*}
\left(\left[\!\begin{array}{cccc}
-1 & \,0\, & 0 & 0\\ 0 & 0 & 0 & 1\\ 0 & 0 & f_0 & f_1 \\ 0 & 1 & f_1 & f_2
\end{array}\right],
\left[\begin{array}{cccc}
0 & \,\,\,0\,\,\, & 0 & c\\ 0 & 0 & 1 & 0\\ 0 & 1 & f_1 & 0 \\ c & 0 & 0 & -f_3
\end{array}\right]
\right)
\end{equation*}
for $n=4$; and by
\begin{equation*}
\left(
\left[\!\begin{array}{cccccc}
-1 & \,0\, & \,\,0\,\, & 0 & 0 & 0\\ 
0 & 0 & 0 & 0 & 0 & 1\\ 
0 & 0 & 0 & 0 & 1 & 0\\ 
0 & 0 & 0 & f_0 & f_1 & f_2\\ 
0 & 0 & 1 & f_1 & f_2 & f_3\\ 
0 & 1 & 0 & f_2 & f_3 & f_4\\ 
\end{array}\right],
\left[\!\begin{array}{cccccc}
\, 0\, & \,0\, & \,\,0\,\, & 0 & 0 & c\\ 
0 & 0 & 0 & 0 & 1 & 0\\ 
0 & 0 & 0 & 1 & 0 & 0\\ 
0 & 0 & 1 & f_1 & f_2 & 0\\ 
0 & 1 & 0 & f_2 & f_3 & 0\\ 
c & 0 & 0 & 0 & 0 & \!-f_5\!\\ 
\end{array}\right]\right)
\end{equation*} 
 for $n=6$.  For general $n$, we have
\begin{equation*}\!\!\!\!\!\!\!A=
\left[\!\!\!\!\begin{array}{cccccccccccccc}
-1&0&0 & \cdots &0& 0 & 0 & 0 & 0 & 0 & \cdots & 0 & 0 & 0 \\[.128in]
0&0&0 & \cdots &0& 0 & 0 & 0 & 0 & 0 & \cdots & 0 & 0 & 1  \\[.128in]
0&0&0 & \cdots &0& 0 & 0 & 0 & 0 & 0 & \cdots & 0 & 1 & 0  \\
\phantom{f_{n/2+1}} &\;\;\;\vdots\;\;\;& \phantom{f_{n/2+1}}  &\ddots & \phantom{f_{n/2+1}}   &\;\;\;\vdots\;\;\;& \phantom{f_{n/2+1} }  & \phantom{f_{n/2+1}}   &\;\;\; \vdots\;\;\;  &\phantom{f_{n/2+1}}   & \iddots & \iddots&\;\;\;\vdots\;\;\; &\phantom{f_{n/2+1}}    \\
0&0&0 & \cdots &0& 0 & 0 & 0 & 0 & 0 & \iddots & 0 & 0 & 0  \\[.128in]
0&0&0 & \cdots &0& 0 & 0 & 0 & 0 & 1 & \cdots & 0 & 0 & 0  \\[.128in]
0&0&0 & \cdots &0& 0 & 0 & 0 & 1 & 0 & \cdots & 0 & 0 & 0  \\[.128in]
0&0&0 & \cdots &0& 0 & 0 & f_0 & f_1 & f_2 & \cdots & f_{n/2-3} & f_{n/2-2} & f_{n/2-1}  \\[.128in]
0&0&0 & \cdots &0& 0 & 1 & f_1 & f_2 & f_3 & \cdots & f_{n/2-2} & f_{n/2-1} & f_{n/2}  \\[.128in]
0&0&0 & \cdots &0& 1 & 0 & f_2 & f_3 & f_4 & \cdots & f_{n/2-1} & f_{n/2} & f_{n/2+1}  \\
 &\vdots&  & \iddots & \iddots  &\vdots &   &   &\vdots   &  &\ddots &  & \vdots &   \\
 0&0&0 & \iddots &0& 0 & 0 & f_{n/2-3} & f_{n/2-2} & f_{n/2-1} & \cdots & f_{n-6} & f_{n-5} & f_{n-4}  \\[.128in]
0&0&1 & \cdots &0& 0 & 0 & f_{n/2-2} & f_{n/2-1} & f_{n/2} & \cdots & f_{n-5} & f_{n-4} & f_{n-3}  \\[.128in]
0&1&0 & \cdots &0& 0 & 0 & f_{n/2-1} & f_{n/2} & f_{n/2+1} & \cdots & f_{n-4} & f_{n-3} & f_{n-2}  \\
 \end{array}\!\right]  
\end{equation*}

\vspace{.1in}
\begin{equation*}
\!\!\!\!\!\!\!\!B=
\left[\!\!\!\!\begin{array}{cccccccccccccc}
0&0&0 &  &0& 0 & 0 & 0 & 0 & 0 &  & 0 & 0 & c \\[.128in]
0&0&0 & \cdots &0& 0 & 0 & 0 & 0 & 0 & \cdots & 0 & 1 & 0  \\[.128in]
0&0&0 &  &0& 0 & 0 & 0 & 0 & 0 &  & 1 & 0 & 0  \\
\phantom{f_{n/2+1}} &\;\;\;\vdots\;\;\;& \phantom{f_{n/2+1}}  & \ddots & \phantom{f_{n/2+1}}   &\;\;\;\vdots\;\;\;& \phantom{f_{n/2+1} }  & \phantom{f_{n/2+1}}   &\;\;\; \vdots\;\;\;  &\phantom{f_{n/2+1}}   & \iddots & \phantom{f_{n/2+1}} &\;\;\;\vdots\;\;\; &\phantom{f_{n/2+1}}    \\
0&0&0 &  &0& 0 & 0 & 0 & 0 & 1 &  & 0 & 0 & 0  \\[.128in]
0&0&0 & \cdots &0& 0 & 0 & 0 & 1 & 0 & \cdots & 0 & 0 & 0  \\[.128in]
0&0&0 &  &0& 0 & 0 & 1 & 0 & 0 &  & 0 & 0 & 0  \\[.128in]
0&0&0 &  &0& 0 & 1 & f_1 & f_2 & f_3 &  & f_{n/2-2} & f_{n/2-1} & 0  \\[.128in]
0&0&0 & \cdots &0& 1 & 0 & f_2 & f_3 & f_4 & \cdots & f_{n/2-1} & f_{n/2} & 0  \\[.128in]
0&0&0 &  &1& 0 & 0 & f_3 & f_4 & f_5 &  & f_{n/2} & f_{n/2+1} & 0 \\
 &\vdots&  &\iddots  &   &\vdots &   &   &\vdots   &  & \ddots &  & \vdots &   \\
 0&0&1 &  &0& 0 & 0 & f_{n/2-2} & f_{n/2-1} & f_{n/2} &  & f_{n-5} & f_{n-4} & 0  \\[.128in]
0&1&0 & \cdots &0& 0 & 0 & f_{n/2-1} & f_{n/2} & f_{n/2+1} & \cdots & f_{n-4} & f_{n-3} & 0  \\[.128in]
c&0&0 &  &0& 0 & 0 & 0 & 0 & 0 &  & 0 & 0 &- f_{n-1}  \\
 \end{array}\!\right]  
\end{equation*}
\vspace{.025in}

\noindent 
In each case, we have
$$(-1)^{n/2}\det(Ax-By) = f(x,y),$$
giving the desired pair $(A,B)$ of symmetric matrices over $\Z$ with invariant form $f$ associated to the integral point $(0,1,c)$ on $z^2=f(x,y)$.  

\vspace{.1in}
If we use $\SL_n^{\pm}$ to denote the subgroup of $\GL_n$ consisting of elements of determinant $\pm1$, then since the action of $\SL_n^\pm$ on $\Z^2\otimes\Sym_2\Z^n$ preserves the invariant binary $n$-form, our construction above yields an entire $\SL_n^\pm(\Z)$-orbit of elements in $\Z^2\otimes\Sym_2\Z^n$ having invariant from $f$ associated to a rational point on $z^2=f(x,y)$.  
However, 
the above construction of an orbit in $\Z^2\otimes\Sym_2\Z^n$ from an integral solution $(x_0,y_0,c)$ to $z^2=f(x,y)$, where $x_0$ and $y_0$ are relatively prime, could depend a priori on the choice of element $\gamma\in\SL_2(\Z)$ taking $(x_0,y_0)$ to $(0,1)$.  
We claim that 
the  $\SL_n(\Z)$-orbit
that we have constructed is in fact well defined and  independent of this choice of $\gamma$.  

To see this, note that the stabilizer
$\{\gamma\in\SL_2(\Z):(0,1)\cdot\gamma = (0,1)\}$
consists of elements of the form $\gamma=\bigl(\begin{smallmatrix}
a & b\\ c& d\end{smallmatrix}\bigr)$ where $a=d=1$ and $c=0$.  For a
given value of $b\in\Z$, the orbit constructed from the solution $(0,1,z_0)$ to $z^2=f'(x,y)=(\gamma\cdot f)(x,y)$ corresponds to the pair $(I',\theta')$ for $R_{f'}$, where
\begin{equation}
I'=\langle c,\theta' I_{f'}^{\textstyle\frac{n-4}2}\rangle
\end{equation}
and
$$
\theta'=\frac{d\theta-c}{-b\theta+a}=\frac{\theta}{-b\theta+1}.
$$
That is, $I'^2 \subseteq \theta' I_{f'}^{n-3}$ and $N(I'^2)=N(\theta')N(I_{f'}^{n-3})$.  Due to the isomorphism (\ref{ifif}), the pair $(I',\theta')$ for $R_{f'}$ corresponds to the pair 
\begin{equation}\label{ifprimecomp}
(I',\theta'/(-b\theta+1)^{n-3})=(I',\theta/(-b\theta+1)^{n-2})\sim
    ((-b\theta+1)^{\textstyle\frac{n-2}2}I',\theta)
\end{equation}
for $R_f$.  Now
$$(-b\theta+1)^{\textstyle\frac{n-2}2}I' = (-b\theta+1)^{\textstyle\frac{n-2}2}\langle c,\theta' I_{f'}^{\textstyle\frac{n-4}2}\rangle = \langle (-b\theta+1)^{\textstyle\frac{n-2}2}c, \theta I_f^{\textstyle\frac{n-4}2} \rangle=\langle c,\theta I_f^{\textstyle\frac{n-4}2}\rangle=I,
$$
so (\ref{ifprimecomp}) becomes simply $(I,\theta)$.  We concude that the $\SL_n(\Z)$-orbit arising from $\gamma\cdot f$ is the same as that coming from~$f$. 
Finally, note that changing $c$ to $-c$ in all these constructions also corresponds to a change of basis in $\SL_n^\pm$ to the basis of $I$.  
In summary, 
the $\SL_n^\pm(\Z)$-orbit that we have constructed
is independent of the particular
$\SL_2(\Z)$-transformation used to send $(x_0,y_0)$ to $(0,1)$ and
depends only on the integral solution $(x_0,y_0,\pm z_0)$ to $z^2=f(x,y)$. 

\vspace{.1in}
Finally, we note that all the definitions, constructions, and arguments we have given in this section apply equally well when working over a general principal ideal domain $D$ in place of $\Z$; we must simply change every occurence of ``$\Z$'', ``integral'', and ``rational'' with ``$D$'', ``$D$-'', and ``$k$-'', respectively, where $k$ denotes the quotient field of $D$.  In particular, the proofs of Theorems~\ref{wood2nn} and \ref{woodsym} with these changes carry over with essentially no change.  (With modified, though somewhat more subtle, statements, Wood in fact proves versions of these theorems over any base scheme; see~\cite{Wood1}.)

Also, via the same constructions, to any $D$-solution $(x_0,y_0,\pm z_0)$ to $z^2=f(x,y)$, where $f$ is a binary $n$-ic form over $D$ and $x_0$ and $y_0$ are relatively prime in $D$ (i.e., generate the unit ideal),
we may associate a well-defined $\SL_n^\pm(D)$-orbit in $D^2\otimes\Sym_2D^n$ having
invariant binary form~$f$.  
The latter construction of an orbit from a solution $(x_0,y_0,\pm z_0)$ over $D$,
where $x_0$~and~$y_0$ are relatively prime, can in fact be carried out over any ring $D$, using the general explicit formula for~$(A,B)$.

\section{Orbits over fields}\label{fieldorbits}

We now examine more carefully the orbits of 
$G(k)=\SL_n^\pm(k) 
:=\{g\in\GL_n(k)\,|\,\det(g)=\pm1\}$ 
on $V(k)=k^2\otimes \Sym_2k^n$ over a general field~$k$.   
We will be particularly interested in the arithmetic fields $k=\R$ and $k=\F_p$.  For any field $k$, we may still speak of the invariant binary $n$-ic form $f$ over $k$ associated to an element of $V(k)$ and the associated $k$-algebra $K_f=R_f$ of dimension $n$ over $k$ (which are constructed
in the identical manner).

\subsection{Classification of orbits over a general field}\label{genfield}

Over a field $k$, we have that
elements of $k^2\otimes\Sym_2k^n$ having a given separable invariant
binary $n$-ic form $f$ with nonzero leading coefficient $f_0$ 
correspond to equivalence classes of pairs $(I,\alpha)$,
where $I$ is a based ideal of $K_f$ such that $I^2=\alpha
I_f^{n-3}$ as ideals and $N(I)^2= N(\alpha)N(I_f^{n-3})$.
Since $k$ is a field, ideals $I$ in $K_f$ are all equal to $K_f$,
and so we may drop $I$ from the data and simply retain its basis.
Thus we see that elements of $k^2\otimes\Sym_2k^n$ parametrize pairs $(\BB,\alpha)$,
where $\BB$ is a basis of $K_f$, $\alpha\in K_f$, and
$N(\BB)^2=N(\alpha)/f_0^{n-3}$, up to the equivalence:
$(\BB,\alpha)\sim (\kappa \BB,\kappa^2\alpha)$ for any $\kappa\in
K_f^\times$.

The changes of basis on $\BB$ that preserve $N(\BB)^2$, and thus the condition $N(\BB)^2=N(\alpha)/f_0^{n-3}$, are precisely the elements of $G(k)$, leading to our
action of $G(k)$ on $k^2\otimes\Sym_2k^n$.  Setting $N=N(\BB)$ (which is thus well-defined up to a factor of $-1$ under the action of $G(k)$), we then see that $G(k)$-orbits on $k^2\otimes\Sym_2k^n$ parametrize equivalence classes of pairs
$(N,\alpha)$ satisfying $N^2 = N(\alpha)/f_0^{n-3}$, where $(N,\alpha)$ is equivalent to $(N(\kappa)\cdot N, \kappa^2\alpha)$ for any $\kappa\in K_f$. 

Since in any such valid pair $(N,\alpha)$, the element $N\in k^\times$ is determined by $\alpha$ (again, up to a factor of $-1$), we conclude that $G(k)$-orbits on $k^2\otimes\Sym_2k^n$ correspond to elements $\alpha\in K_f^\times/K_f^{\times2}$ mapping to the class of $f_0^{n-3}$ ($=$ the class of $f_0$) in $k^\times/k^{\times2}$.  We denote this set of $\alpha$ by $(K_f^\times/K_f^{\times2})_{N\equiv f_0}$.

Explicitly, the orbit associated to such an $\alpha\in(K_f^\times/K_f^{\times2})_{N\equiv f_0}$ is that of the pair $(A,B)$ of symmetric $n\times n$ matrices whose entries are given by taking the coefficients of $\zeta_{n-1}$ and $\zeta_{n-2}$ in the image of the map 
\begin{equation}\label{mt}
\frac1\alpha\times:\BB\times \BB\to K_f,
\end{equation} where $\BB$ is any basis of $K_f$ satisfying 
$N(\BB)^2=N(\alpha)/f_0^{n-3}$.  From this description, we then
observe that a 
change of basis to $\BB$, via an element of $\SL_n^\pm(k)$, that
preserves 
the multiplication table~(\ref{mt}) corresponds to multiplying the
basis $\BB$ by an element $\kappa\in K_f^\times$ such that $\kappa^2=1$. 

We summarize this discussion as follows:

\begin{theorem}
\label{woodsymfield}
The orbits of $\SL_n^\pm(k)$ on
$k^2\otimes\Sym_2k^n$ having a given separable invariant binary
$n$-ic form $f$ over $k$ with nonzero leading coefficient $f_0$
are in canonical bijection with elements of                     
$(K_f^\times/K_f^{\times2})_{N\equiv f_0}$. 
The stabilizer of such an element of $k^2\otimes\Sym_2k^n$  is
isomorphic to 
$K_f^\times[2]$.
\end{theorem}

We may now ask which orbits, when viewed as elements of 
$(K_f^\times/K_f^{\times2})_{N\equiv f_0}$, come from $k$-rational
solutions to $z^2=f(x,y)$.  
The construction of the previous section shows that if $(0,1,z_0)$ is a $k$-solution to $z^2=f(x,y)$, then the element of $(K_f^\times/K_f^{\times2})_{N\equiv f_0}$ corresponding to the solution 
$(0,1,z_0)$ 
is given by the class of $\theta$, where
$\theta$ denotes as before the root of $f(x,1)$ that is used to define
$R_f=K_f$ via (\ref{zetadef}).  

Let us next
determine the element of $(K_f^\times/K_f^{\times2})_{N\equiv f_0}$ corresponding to a
general $k$-solution 
of $z^2=f(x,y)$.  First, note that if
$(x_0,y_0,z_0)$  is a $k$-solution to $z^2=f(x,y)$, then $(0,1,z_0)$ is a $k$-solution to $z^2=f'(x,y)$, where $f'=\gamma\cdot  f$ with
$\gamma= \bigl(\begin{smallmatrix}
a & b\\ c& d\end{smallmatrix}\bigr)=
\bigl(\begin{smallmatrix}
s & \!\!-r\\ x_0& y_0\end{smallmatrix}\bigr),
$
and $r$ and $s$ are any constants in $k$ such that $rx_0+sy_0=1$.  
The $k$-solution $(0,1,z_0)$ to $z^2=f'(x,y)$ corresponds to the class of 
$\theta' = (d\theta-c)/(-b\theta+a)=(y_0\theta-x_0)/(r\theta+s)$ in $(K_f^\times/K_f^{\times2})_{N\equiv f_0'}$.  The canonical map 
$(K_{f'}^\times/K_{f'}^{\times2})_{N\equiv f'_0}\to 
(K_f^\times/K_f^{\times2})_{N\equiv f_0}$
is given by multiplication by
$(-b\theta+a)^{-(n-3)}\equiv r\theta+s$. We conclude that 
the $k$-solution $(x_0,y_0,z_0)$ to $z^2=f(x,y)$ corresponds to the element
$y_0\theta-x_0\in (K_f^\times/K_f^{\times2})_{N\equiv f_0}$.

In particular, if the $k$-orbit constructed in Section~\ref{rptovz} associated to
a $k$-rational solution $(x_0,y_0,z_0)$ to $z^2=f(x,y)$ is given by 
$\alpha\in(K_f^\times/K_f^{\times2})_{N\equiv f_0}$, then the $k$-orbit associated to
$(rx_0,ry_0,r^{n/2}z_0)$ for $r\in k^\times$ is given by $r\alpha\in(K_f^\times/K_f^{\times2})_{N\equiv f_0}$.  It follows that the set of $G(k)$-orbits corresponding to a $k$-rational {point} $P=(x_0:y_0:z_0)\in\P(1,1,n/2)$ on the hyperelliptic curve  $C:z^2=f(x,y)$ forms a torsor for $k^\times/(k^\times\cap K_f^{\times2})$.
The resulting map $\mu:C(k)
\to (K_f^\times/(K_f^{\times2}k^\times))_{N\equiv f_0}$, defined by sending $(x_0:y_0:z_0)$ to the class of $y_0\theta-x_0$, is the well-known ``$x-T$'' map of Cassels!

\subsection{Orbits over $\R$}\label{rorbits}

We may use Theorem~\ref{woodsymfield} 
to classify the $G$-orbits 
 over $\R$
having a given separable invariant binary $n$-ic form $f$ over $\R$.
If $f$ is negative definite, then there are no orbits of $G(\R)$ 
on $V(\R)$ having invariant binary form $f$ (since the norms from
$K_f\cong \C^{n/2}$ to $\R$ are all positive, while the leading
coefficient of $f$ is negative).

Otherwise, suppose $f$ is not negative definite and has $2m$ real
roots ($m\in\{0,\ldots,n/2\}$).  By an $\SL_2$-change of basis,
we may assume that the leading coefficient $f_0$ of $f$ is nonzero.  
Then $K_f\cong \R^{2m}\times
\C^{n/2-m}$.  The number of elements in
$(K_f^\times/K_f^{\times2})_{N\equiv f_0}$ is given by 1 if
$m=0$, and otherwise is $2^{2m-1}$. Namely, these elements correspond
to the possible patterns of $2m$ signs of the images of these elements
in $\R^{2m}$, subject to the condition that the number $n_-$ of minus
signs satisfies $(-1)^{n_-}=\mbox{sgn}(f_0)$.  

Thus, by Theorem~\ref{woodsymfield}, if $f$ is not negative definite
and has $2m$ real roots, then there are $\lceil 2^{2m-1}\rceil$ 
orbits
of $G(\R)$ 
on $V(\R)$ having invariant binary form $f$.  Moreover, 
the stabilizer in $G(\R)$ 
of such an orbit
is isomorphic to $K_f^\times[2]\cong (\Z/2\Z)^{n/2+m}$.

\subsection{Orbits over $\F_p$}\label{fporbits}

Next, we turn to the orbits of $G(\F_p)$ 
on $V(\F_p)$ having a given
separable invariant binary $n$-ic form $f$ over $\F_p$ with nonzero
leading coefficient $f_0$.  Suppose
that $f(x,y)$ factors into irreducible polynomials of degrees
$d_1,\ldots,d_m$, where $d_1+\cdots+d_m=n$.  Then $K_f\cong
\prod_{i=1}^m \F_{p^{d_i}}$.  Hence
$|(K_f^\times/K_f^{\times2})_{N\equiv f_0}| = 2^{m-1}$ if $p\neq 2$
and is trivial otherwise.  

It follows that the number of
$G(\F_p)$-orbits on $V(\F_p)$ having invariant binary form $f$ is
$2^{m-1}$ if $p\neq 2$ and is 1 otherwise.  Moreover, the stabilizer
in $G(\F_p)$ of such an orbit is isomorphic to $K_f^\times[2]$, which
in turn is isomorphic to $(\Z/2\Z)^{m}$ if $p\neq 2$ and is trivial
otherwise.  That is, if $p\neq2$, then the number of $G(\F_p)$-orbits yielding a
given separable $f$ is always equal to half the size of the
stabilizer.  In particular, when $p\neq 2$, the total number of
elements in $V(\F_p)$ having given separable invariant binary form $f$
is $\#G(\F_p)/2=\#\SL_n(\F_p)$, independent of $f$.  On the other
hand, since the number of orbits is equal to the size of the
stabilizer when $p=2$, the total number of elements in $V(\F_2)$
having given separable invariant binary form $f$ is $\#G(\F_2)$, which
is again equal to $\#\SL_n(\F_2)$.

Thus, for any $p$, the total number of elements in $V(\F_p)$
having given separable invariant binary form $f$ with nonzero leading
coefficient $f_0$ is given by $\#\SL_n(\F_p)$.

\subsection{The densities of local orbits coming from rational points
  on hyperelliptic curves}\label{secdensity}

Let $\mathcal P$ denote the union of the $G(\Z)$-orbits of all
elements in $V(\Z)$ that come from rational points on hyperelliptic
curves via the construction in Section~\ref{rptovz}; similarly, for each place
$\nu$ of $\Q$, let $\mathcal P_\nu$ denote the union of the $G(\Z_\nu)$-orbits of all
elements in $V(\Z_\nu)$ that come from $\Q_\nu$-rational points on hyperelliptic
curves via the construction in Section~\ref{rptovz}.  Then $\mathcal P\subset \mathcal P_\nu$ for all $\nu$.
In this subsection, we study the subset
of $G(\Z_\nu)$-orbits on $V(\Z_\nu)$ that lie in $\mathcal P_\nu$.

We consider first $\nu=\infty$.  Let $f$ be a binary $n$-ic form over
$\R$ with $2m$ real roots that is not negative definite.  We wish to
know how many of the $\lceil 2^{2m-1}\rceil$ nondegenerate orbits of
$G(\R)$ on $V(\R)$ can arise from a rational solution to $z^2=f(x,y)$
via the construction of Section~\ref{rptovz}.  
If $m=0$, then clearly there is only one such orbit.  If $m\geq 1$, then 
by a linear change of variable on the binary form $f$ if necessary, we may write
$f(x,1)=f_0\prod_{i=1}^{2m}g_i(x)\prod_{j=1}^{n/2-m}h_j(x)$, where $f_0<0$,
the $g_i$'s are monic linear, and the $h_j$'s are irreducible monic
quadratic polynomials over $\R$. We order the $g_i$'s by increasing
order of their roots in $\R$.  Then any point on the hyperelliptic curve 
$C:z^2=f(x,y)$ can be expressed as $(a:1:c)$ where $a,c\in\R$.  
The real orbit (via the construction of
Section~\ref{rptovz}) arising from the solution $(a,1,c)$ to $z^2=f(x,y)$
has corresponding pattern of signs $+\cdots+\;-\cdots -$, 
where the number of initial plus signs is
given by the number of $g_i$'s having roots smaller than $a$.  There
are thus exactly $m$ possible patterns of signs where the number
of minus signs is odd.  The negatives of these patterns of signs also can arise,
by considering instead the solutions $(-a,-1,c)$ of $z^2=f(x,y)$. (Indeed, recall that if $f$ is not definite, then each rational point on 
$C$ corresponds to $|\R^\times/( \R^\times\cap K_f^{\times2})|=2$ real orbits!) 
It follows that precisely $2m$ of the $\lceil 2^{2m-1}\rceil$
nondegenerate orbits of  
$G(\R)$ on $V(\R)$ having invariant binary form $f$ lie in $\mathcal P_\infty$ when $m\geq 1$.  

Next, we turn to $\mathcal P_p$ for a finite prime~$p\neq2$.  Let
$\bar{\mathcal P}_p$ denote the image of $\mathcal P_p$ in $V(\F_p)$.
Suppose $f$ is a (not necessarily separable) binary $n$-ic form over
$\F_p$ 
having $m$ distinct irreducible factors over $\F_p$.  
If $f$ is separable, then we have seen that
there are $2^{m-1}$ 
orbits of $G(\F_p)$ on $V(\F_p)$ having
invariant binary form~$f$.  Let $k\leq p+1$ denote the number of
$(a:b)\in \P^1(\F_p)$ such that $f(a,b)$ is a square.   The number of 
$G(\F_p)$-orbits on $V(\F_p)$ arising from such an $(a:b)$, via the
construction of Section~\ref{rptovz}, is 
$$|\F_p^\times/(\F_p^\times \cap K_f^{\times2})|=
\left\{\begin{array}{ll} 
1 & \mbox{if all factors of $f$ over $\F_p$
      have even degree;}\\
2 & \mbox{otherwise}. 
\end{array}\right.
$$
Thus at most $2k$ 
orbits of $G(\F_p)$ on $V(\F_p)$ having invariant
binary form $f$ can lie in $\bar{\mathcal P}_p$.  
By construction, the cardinality of the stabilizer of an element in any one of these orbits is at least $|K_f[2]|=2^m$, and is exactly $2^m$ if $f$ is separable.  
Noting that for a binary form $f$ over $\Q_p$ and a nonresidue $u$ modulo~$p$, not both $f$ and $uf$ can take square values at a given point $(a,b)$.  Hence, for each $m$, we may use $k_0=(p+1)/2$ as an upper estimate for $k$ on average.  
Let $I_p(m)$ denote the set of binary $n$-ic forms modulo
$p$ having $m$ distinct irreducible factors. 
It follows that 
the $p$-adic density of $\mathcal P_p$ in $V(\Z_p)$ is at most 
$$\sum_{m=0}^n
\min\left\{1,\frac{p+1}{2^{m-1}}\right\}
\frac{\#\SL_n(\F_p)}{p^{n(n+1)}}|I_p(m)|.
$$
We have not justified the factor of $\#\SL_n(\F_p)$ in the case where the binary form $f$ is not separable with $k>2^{m-1}$, but the correctness of this factor can be deduced also in this case by working over $\Q_p$ and $\Z_p$ instead of $\F_p$, as will be explained in Section~\ref{mainproof}.  Since nonseparable forms over $\F_p$ occur with density only $O(1/p)$, this issue will not have any noticeable effect on our eventual application anyhow, so we do not get into too much detail on this point.

To obtain a bound on $\mathcal P_2$, it is not effective to work modulo 2, because all integers are squares modulo 2!  So we work instead over $\Z/8\Z$.  
Let $\bar{\mathcal P}_2$ denote the image of $\mathcal P_2$ in $V(\Z/8\Z)$.  Suppose $f$ is a binary $n$-ic form over $\Z/8\Z$ 
having a factorization with $m$ distinct irreducible factors over $\Z/2\Z$.  Then the number of elements in $R_f^\times[2]$ is at least $2^{n+m}$.  Let $k\leq 12=|\P^1(\Z/8\Z)|$ denote the number of
$(a:b)\in \P^1(\Z/8\Z)$ such that $f(a,b)$ is a square.   The number of 
$G(\Z/8\Z)$-orbits on $V(\Z/8\Z)$ arising from such an $(a:b)$, via the
construction of Section~\ref{rptovz}, is at most $|(\Z/8\Z)^\times/(\Z/8\Z)^{\times2}|\cdot k=4k$.  Noting that at most one of $f$, $3f$, $5f$, and $7f$ can take a square value at a given point $(a:b)$, we see that we may use $k_0=12/4 = 3$ as an upper estimate for $k$ on average.  Let $I_8(m)$ denote the set of binary $n$-ic forms modulo
$8$ having a factorization with $m$ distinct irreducible factors modulo 2. It follows that 
the $2$-adic density of $\mathcal P_2$ in $V(\Z_2)$ is at most 
$$\sum_{m=0}^n
\min\left\{1,\frac{12}{2^{n+m-1}}\right\}\frac{\#\SL_n(\Z/8\Z)
}{8^{n(n+1)}}|I_8(m)|.$$

\section{Counting the number of orbits of bounded height}\label{counting}

Recall that we write hyperelliptic curves $C$ of genus $g$ over $\Q$
in the form
\begin{equation}\label{hypereq2}
C: z^2 = f(x,y) = f_0x^n+f_1x^{n-1}y+\cdots+f_ny^n,
\end{equation}
where $n=2g+2$ and $f$ is {\it nondegenerate}, i.e., $f$ factors into
distinct linear factors over $\bar\Q$ (equivalently, $f$ is
squarefree or the discriminant $\Delta(f)$ of $f$ is nonzero).  We define the
{\it height} $H(C)$ of $C$ by
$$ H(C):= H(f):=\max\{|f_i|\}.$$
The number of such hyperelliptic curves $C$ (or integral binary forms
$f$) having height less than $X$ is thus $\sim (2X)^{n+1}$.

For a ring $T$, let $G(T):=\SL_n^\pm(T)$ and $G'(T):=\GL_n(T)$, and set
$V(T)=T^2\otimes\Sym_2T^n$.  Then the groups $G(T)$ and $G'(T)$ both act on $V(T)$. (Note that 
$G(\Z)=G'(\Z)=\GL_n(\Z)$.)  Given an element $v=(A,B)\in V_T$, we define the
binary form $f_v$ over $T$ by $f_v(x,y):=(-1)^{n/2}\det(Ax-By)$, and
say that $v$ is {\it nondegenerate} if $f_v$ is (i.e., $f_v$ does not factor over $\Q$).  Furthermore, we say
that $v$ is {\it irreducible} if $f_v$ is. Finally, we define the
discriminant and height of an element $v\in V(\Z)$ by $\Delta(v):=\Delta(f_v)$ and $H(v):=H(f_v)$.

In this section, we wish to determine the number of 
nondegenerate
$G(\Z)$-orbits on $V(\Z)$ having height less than $X$.  We will prove the following theorem:
\begin{theorem}\label{countvz}
There exists a positive constant $\kappa$ such that the number of 
nondegenerate
$G(\Z)$-orbits on $V(\Z)$ having height less than $X$ is $\kappa\cdot X^{n+1} + o(X^{n+1})$.
\end{theorem}
\noindent
We will show that the statement of Theorem~\ref{countvz} also remains true if ``nondegenerate'' is replaced by ``irreducible''; in other words, the number of reducible orbits of height less than $X$ is negligible.

Since the total number of integral binary forms of height at most $X$ is $\sim 2^{n+1}X^{n+1}$, Theorem~\ref{countvz} implies that the average number of $G(\Z)$-orbits on $V(\Z)$ having a given invariant form $f$, over all nondegenerate integral binary forms $f$, is bounded by a constant %
(depending only on $n$).

\subsection{Construction of fundamental domains}

In this subsection, we construct convenient fundamental domains for
the action of $G(\Z)$ on the set of nondegenerate elements of $V(\R)$. In the next subsection, we will then count the number of {} points in $V(\Z)$ in these
fundamental domains having bounded height, proving
Theorem~\ref{countvz}, which will then allow us to prove Theorem~1 in
Section~\ref{mainproof}.

To describe our fundamental domains explicitly, we put natural
coordinates on $V$.  Namely, $V$ may be identified with the
$n(n+1)$-dimensional space of pairs of symmetric matrices.  For
$(A,B)\in V$, we use $a_{ij}$ and $b_{ij}$ to denote the $(i,j)$
entries of $A$ and of $B$, respectively; thus the $a_{ij}$ and
$b_{ij}$ $(1\leq i\leq j\leq n)$ form a natural set of coordinates on
$V$.

\subsubsection{Fundamental sets for the action of $G'(\R)$ on $V(\R)$}\label{fundsets}

By the discussion in \S\ref{rorbits}, we may naturally partition the
set of elements in $V(\R)$ with $\Delta\neq 0$ into $\sum_{m=0}^{n/2}
\lceil 2^{2m-1}\rceil$ components, which we denote by $V^{(m,\tau)}$
for $m=0,\ldots,n/2$ and $\tau=1,\ldots,\lceil2^{2m-1}\rceil$.  For a
given value of~$m$, the component $V^{(m,\tau)}$ in $V(\R)$ maps to
the component $I(m)$ of the space of binary $n$-ic forms in $\R^{n+1}$ having
nonzero discriminant and $2m$ real roots. The parameter $\tau$ runs
over all assignments of signs $\pm$ to the $2m$ real roots, subject to
the condition that the number of $-$ signs is odd or even in
accordance with whether the leading coefficient of the invariant
binary $n$-ic form is negative or positive.  
We have seen that the stabilizer in $G(\R)$ (or $G'(\R)$) of an element $v$ in
$V^{(m,\tau)}$ has size $2^{n/2+m}$.

We first construct fundamental sets $L^{(m,\tau)}$ for the action of
$G'(\R)$ on $\{(A,B)\in V^{(m,\tau)}\}$ that consist of elements of
height 1 and are {\it bounded} sets in $V(\R)$.  
We begin with the case when $m = n/2$.
In this case, we claim that we can take $L^{(m,\tau)}$ to consist of
pairs $(A,B)$ of diagonal matrices.  Indeed, any binary $n$-ic
form in $I(m)$ has $n$ roots in $\P^1(\R)$, say $r_i=(\mu_i:\lambda_i)$.  We
may normalize $\mu_i,\lambda_i$ so that $\lambda_i>0$, $|\mu_i|,\lambda_i\leq 1$ and at least one
of $|\mu_i|,\lambda_i$ equals 1.  We then put a total ordering on such elements
$(\mu_i:\lambda_i)$ in $\P^1(\R)$ lexicographically.  We may then assume 
that $r_1<\cdots<r_n$ with respect to this ordering. 

Given such a sequence $r_1<\cdots<r_n$ in $\P^1(\R)$, where
$r_i=(\mu_i:\lambda_i)$ as above, we construct a pair $(A,B)$ of $n\times n$
real symmetric matrices, where $A$ is the diagonal matrix $(\epsilon_1
\lambda_1,\ldots,\epsilon_n \lambda_n)$ and $B$ is the diagonal matrix
$(\epsilon_1 \mu_1,\ldots,\epsilon_n \mu_n)$, where each
$\epsilon_i=\pm 1$.  Note that $\det(Ax-By)$ then has roots
$r_1,\ldots,r_n$ in $\P^1(\R)$.  The set of all such matrices $(A,B)$, over all
such sequences  $r_1\leq \cdots\leq r_n$, forms a compact set and all elements
in this set have nonzero height, by construction.
In particular, the heights of all such pairs of symmetric $n\times n$ 
matrices are bounded above and
are bounded below away from zero.  Thus if we scale all such pairs
$(A,B)$ to have height 1, over all $r_1<\cdots<r_n$ in $\P^1(\R)$, we are still in a
bounded set, which we denote by~$L(m)$.  

Now each binary $n$-ic form $f\in I(m)$ with nonzero leading
coefficient $f_0$ and height 1 occurs as the invariant binary $n$-ic form of
exactly $2^{n-1}=2^{2m-1}$ elements $(A,B)$ in $L(m)$, corresponding to the
possible choices of $\epsilon_1,\ldots,\epsilon_n$ such that
$\epsilon_1\cdots\epsilon_n=\mbox{sgn}(f_0)$. Moreover, these $2^{2m-1}$
choices of $(A,B)$ 
lie in distinct $G'(\R)$-orbits, as they yield the $2^{2m-1}$ possible patterns
of signs in $K_f\cong \R^{2m}$. 
So all the 
$G'(\R)$-orbits on $V(\R)$ having invariant binary $n$-ic form equal
to $f$ are represented in $L(m)$.  Thus, when
$m=n/2$, for each $\tau$ we may simply take $L^{(m,\tau)}:=L(m)\cap V^{(m,\tau)}$, and
this gives our desired bounded fundamental set for the action of
$G'(\R)$ on $V^{(m,\tau)}$ consisting of elements of height~1.

We may proceed in a similar manner when $m < n/2$. Any
binary $n$-ic form in $I(m)$ has $2m$ real roots and $n/2-m$ pairs of complex
conjugate roots, for a total of $n/2+m$ roots $r_1,\ldots,r_{2m}$ and $r_{2m+1},\ldots,r_{n/2+m}$
in $\P^1(\R)$ and the upper half plane of $\P^1(\C)$, respectively.  We
assume again that $r_1<\cdots<r_{2m}$.  Furthermore, we may assume
that $(r_{2m+1},\ldots,r_{n/2+m})$ lies in a fundamental domain for the action of
the symmetric group $S_{n/2-m}$ on the product of $n-m$ upper half
planes minus all the diagonals.  For all $i$, let us write again 
$r_i=(\mu_i:\lambda_i)$, where $\lambda_i>0$, $|\mu_i|,\lambda_i\leq 1$ and at least one
of $|\mu_i|,\lambda_i$ equals 1.  Also, given a complex number
$z=x+y\sqrt{-1}$, set $\psi(z)=\bigl(\begin{smallmatrix}x&y\\-y&x 
\end{smallmatrix}\bigr)$.  Then we may consider all matrices $(A,B)$
where $A$ is the block diagonal matrix
$$(\epsilon_1\lambda_1,\,\ldots\,,\epsilon_{2m}\lambda_{2m},\,\psi(\lambda_{2m+1}),\,\ldots\,,
\psi(\lambda_{n/2+m}))$$
and 
$B$ is the block diagonal matrix
$$(\epsilon_1\mu_1,\,\ldots\,,\epsilon_{2m}\mu_{2m},\,\psi(\mu_{2m+1}),\,\ldots\,,
\psi(\mu_{n/2+m}))$$
where again each $\epsilon_i=\pm1$. 
Then the binary $n$-ic form
$\det(Ax-By)$ has the desired roots $r_1,\ldots,r_{2m}$ and $r_{2m+1},\ldots,r_{n/2+m}$
in $\P^1(\R)$ and the upper half plane, respectively.
The set of all such matrices~$(A,B)$, over all
such sequences $r_1,\ldots,r_{n/2+m}$,
again lie in a bounded set where the heights of all such pairs of
symmetric $n\times n$ matrices are bounded above and
bounded below away from zero.  If we now scale all such pairs
$(A,B)$ to have height 1, over all $r_1,\ldots,r_{n/2+m}$, we obtain a
bounded set, which we denote by~$L(m)$.  We again set
$L^{(m,\tau)}:=L(m)\cap V^{(m,\tau)}$.

\subsubsection{
Fundamental sets for the action of $G'(\Z)$ on $G'(\R)$}

In \cite{BoHa}, Borel and Harish-Chandra construct a natural fundamental set $\FF$ in $G'(\R)$ for the left action of $G'(\Z)$ on $G'(\R)$.  This set $\FF$ may be expressed in the form $$
\FF=
\{ut\theta\lambda:u\in N'(t),\,t\in T',\,\theta\in K,\,\lambda\in\Lambda\},$$ where $N'(t)$ is an absolutely bounded measurable set, which depends on $t\in T'$, of unipotent upper triangular real $n\times n$ matrices; the set $T'$ is the subset of the torus of diagonal matrices with positive entries given by
\begin{equation}\label{nak}
T' = \left\{\left(\begin{array}{cccc}
 t_1^{-1}& & &\\
 & t_2^{-1} & & \\
&  & \ddots &  \\
 & & & t_{n}^{-1} 
\end{array}\right): \;t_1/t_2>c, \;\,t_2/t_3>c, \; \ldots\,,\; t_{n-1}/t_n>c\right\}
\end{equation}
where $t_n=(t_1\cdots t_{n-1})^{-1}$; $K$ is a maximal compact real subgroup of $G(\R)$; and $\Lambda=\{\lambda>0\}$ acts on $V(\R)$ by scaling.
In the above, $c$ denotes an absolute positive constant.  (Note that $t_i/t_{i+1}$ for $i=1,\ldots,n-1$ form a set of simple roots for our choice of maximal torus $T$ in $G^\ss=\SL_n\subset G$.)

It will be convenient in the sequel to parametrize $T'$ in a slightly different way.  Namely, for $m=1,\ldots,n$, we make the substitution 
$$t_m=\prod_{k=1}^{m-1}s_k^{-k}\cdot\prod_{k=m}^{n-1}s_k^{n-k}.$$
Then we may speak of elements $s=(s_1,\ldots,s_{n-1})\in T'$ as well, and the conditions on $s$ for $s=(s_1,\ldots,s_{n-1})$ to be in $T'$ is simply that, for some constants $c_k>0$, we have $s_k>c_k$ for all $k=1,\ldots,n-1$.


\subsubsection{Fundamental sets for the action of $G(\Z)$ on $V(\R)$}

For any $h\in G'(\R)$, by the previous section 
we see that $\FF hL^{(m,\tau)}$ is the union of $2^{n/2+m-1}$
fundamental domains for the action of $G(\Z)$ on $V^{(m,\tau)}$;
here, we regard $\FF hL^{(m,\tau)}$ as a multiset, where the
multiplicity of a point $x$ in $\FF hL^{(m,\tau)}$ is given by the
cardinality of the set $\{g\in\FF\,\,:\,\,x\in ghL^{(m,\tau)}\}$.
Thus, as explained in~\cite[\S2.1]{BS},  a $G(\Z)$-equivalence class $x$ in $V^{(m,\tau)}$ is
represented in this multiset $\sigma(x)$ times, where
$\sigma(x)=\#\Stab_{G(\R)}(x)/\#\Stab_{G(\Z)}(x)$. Since the stabilizer always contains negative the identity in $G(\Z)$, we see that $\sigma(x)$ is always a number between 1
and $2^{n/2+m-1}$.

For any $G(\Z)$-invariant set $S\subset V^{(m,\tau)}\cap V(\Z)$, let $N(S;X)$
  denote the number of $G(\Z)$-equivalence classes of 
  elements $(A,B)\in S$ satisfying $H(A,B)<X$.
  Then we conclude that, for any $h\in G'(\R)$, the product $2^{n/2+m-1}\cdot
N(S;X)$ is exactly equal to the number of integer
points in $\FF hL^{(m,\tau)}$ having height less than $X$,
with the slight caveat that the (relatively rare---see
Proposition~\ref{gzbigstab}) points with $G(\Z)$-stabilizers of
cardinality $2r$ ($r>1$) are counted with weight~$1/r$.

As mentioned earlier, the main obstacle to counting integer points of
bounded height in a single domain $\FF hL^{(m,\tau)}$ is that the
relevant region is not bounded, but rather has cusps going off to
infinity.  We simplify the counting in this cuspidal region by
``thickening'' the cusp; more precisely, we compute the number of
integer points of bounded height in the region $\FF hL^{(m,\tau)}$
by averaging over lots of such fundamental regions, i.e., by averaging
over the domains $\FF hL^{(m,\tau)}$ where $h$ ranges over a certain
compact subset $G_0\in G'(\R)$.  The method, which is an extension of the method of~\cite{dodpf} and \cite{BS}, is described next.

\subsection{Counting irreducible integral points of bounded height}

In this section, we derive asymptotics for the number of
$G(\Z)$-equivalence classes of nondegenerate elements of $V(\Z)$ having bounded invariants.  
We also describe how these
asymptotics change when we restrict to counting only those elements in $V(\Z)$ that satisfy any specified finite set of congruence
conditions.  

Let $V^{(m,\tau)}$ ($m=0,\ldots,n/2$, $\tau=1,\ldots,\lceil2^{2m-1}\rceil$) denote as before the various components of $V(\R)$ of elements having nonzero discriminant. (Recall that the component $V^{(m,\tau)}$ in $V(\R)$ maps to the component $I(m)$ of binary $n$-ic forms in $\R^{n+1}$ having
nonzero discriminant and $2m$ real roots, and the parameter $\tau$ runs
over all assignments of signs $\pm$ to the $2m$ real roots, subject to 
the condition that the number of $-$ signs is odd or even in
accordance with whether the leading coefficient of the invariant
binary $n$-ic form is negative or positive.)
Let 
\[c_{m,\tau} = \frac{\Vol(\FF L^{(m,\tau)}\cap \{v\in V(\R):H(v)<1\})}{2^{n/2+m-1}}.\]
Then in this section we prove the following theorem:

\begin{theorem}\label{thmcount}
 Fix $m,\tau$.  For any $G(\Z)$-invariant set $S\subset V(\Z)^{(m,\tau)}:=V(\Z)\cap V^{(m,\tau)}$, let $N(S;X)$
  denote the number of $G(\Z)$-equivalence classes of nondegenerate
  elements $(A,B)\in S$ satisfying $H(A,B)<X$. Then
$$N(V(\Z)^{(m,\tau)};X)=c_{m,\tau}X^{n+1}+o(X^{n+1}),$$
and
\begin{equation}\label{cmtformula}
c_{m,\tau} = 2^{n/2-m+1}\zeta(2)\zeta(3)\cdots\zeta(n)
\Vol(\{f\in I(m)\subset \R^{n+1}:H(f)<1\}).
\end{equation}
\end{theorem}

\subsubsection{Averaging over fundamental domains}\label{avgsec}

Let $G_0$ be a compact left $K$-invariant set in $G'(\R)$ that is the
closure of a nonempty open set and in which every element has
determinant greater than or equal to $1$. 
Then for any $m,\tau$, we may write
\begin{equation}
N(V(\Z)^{(m,\tau)};X)=\frac{\int_{h\in G_0}\#\{x\in \FF hL\cap V(\Z)
:
  H(x)<X\}dh\;}{2^{n/2+m-1}\cdot\int_{h\in G_0}dh},
\end{equation}
where 
we use $L$ as shorthand $L^{(m,\tau)}$.  The denominator of the
latter expression is an absolute constant $C_{G_0}^{(m,\tau)}$ greater than
zero.

More generally, for any $G(\Z)$-invariant subset $S \subset
V(\Z)^{(m,\tau)}$, let $N(S;X)$ denote the number of 
nondegenerate
$G(\Z)$-orbits in $S$ having height less than $X$. 
Then $N(S;X)$ can be similarly expressed as
\begin{equation}\label{eq9}
N(S;X)=\frac{\int_{h\in G_0}\#\{x\in \FF hL\cap S
: H(x)<X\}dh\;}{C_{G_0}^{(m,\tau)}}.
\end{equation}
We use (\ref{eq9}) to define $N(S;X)$ even for
sets $S\subset V(\Z)$ that are not necessarily $G(\Z)$-invariant.

Now, given $x\in V^{(m,\tau)}$, let $x_L$ denote the {unique} point in $L$
that is $G'(\R)$-equivalent to $x$. We have

\begin{equation}
N(S;X)=\frac{1}{C_{G_0}^{(m,\tau)}}\sum_{\substack{{x\in
      S
      }\\[.02in]{H(x)<X}}}\int_{h\in G_0} \#\{g \in \FF :
x=ghx_L\} dh.
\end{equation}
For a given $x\in S
$, there exist a finite number of elements
$g_1,\ldots,g_k\in G(\R)$ satisfying $g_jx_L=x$.  We then have
$$\int_{h\in G_0} \#\{g \in \FF :x=ghx_L\} dh=\displaystyle\sum_j\int_{h\in G_0} \#\{g \in \FF :gh=g_j\} dh=\displaystyle\sum_j\int_{h\in G_0\cap\FF^{-1}g_j}dh.$$
As $dh$ is an invariant measure on $G$, we have
$$\displaystyle\sum_j\int_{h\in G_0\cap\FF^{-1}g_j}\!\!\!\!\!dh=\displaystyle\sum_j\int_{g\in G_0g_j^{-1}\cap\FF^{-1}}\!\!\!\!\!dg=\displaystyle\sum_j\int_{g\in\FF}\#\{h \in G_0 :gh=g_j\} dg=\int_{g\in \FF} \#\{h \in G_0 :x=ghx_L\}
 dg.$$
Therefore,
\begin{eqnarray}\label{eqavg}
N(S;X)&=&\frac{1}{C_{G_0}^{(m,\tau)}}\,\sum_{\substack{x\in S
\\[.02in]
    H(x)<X}} \int_{g\in\FF} \#\{h \in G_0 : x=ghx_L\}dg.\\
&\!\!=\!\!&  \frac1{C_{G_0}^{(m,\tau)}}\int_{g\in\FF}
\#\{x\in S
\cap gG_0L:H(x)<X\}\,dg\\[.075in] 
&\!\!=\!\!&  \frac1{C_{G_0}^{(m,\tau)}}\int_{g\in N'(s)A'\Lambda K}
\#\{x\in S
\cap u s\lambda \theta G_0L:H(x)<X\}
dg,
\end{eqnarray}
where $dg$ is a Haar measure on $G'(\R)$.  Explicitly, we have
\[dg \,=\,  \prod_{k=1}^{n-1} t_k^{2k-2n} \cdot du\, d^\times t\,d^\times \lambda\, d\theta = 
\prod_{k=1}^{n-1} s_k^{-nk(n-k)}\cdot du\, d^\times s\,d^\times \lambda\, d\theta\,,\]
where $du$ is an invariant measure on the group $N$ of unipotent upper triangular matrices in $G(\R)$, and where we normalize the invariant measure $d\theta$ on $K$ so that $\int_{K} d\theta=1.$


Let us write $E(u,s,\lambda,X) = u s \lambda G_0L\cap\{x\in V^{(m,\tau)}:H(x)<X\}$.
As $KG_0=G_0$ and $\int_K d\theta =1 $, we have
\begin{equation}\label{avg}
N(S;X) = \frac1{C_{G_0}^{(m,\tau)}}\int_{g\in N'(s)A'\Lambda}                              
\#\{x\in S
\cap E(u,s,\lambda,X)\}\prod_{k=1}^n s_k^{-nk(n-k)}
\cdot du\, d^\times s\,d^\times \lambda\,.
\end{equation}
The expression (\ref{avg}) for $N(S;X)$
will be of use frequently in the sections
that follow.

\subsubsection{An estimate from the geometry of numbers}

To estimate the number of lattice points in $E(u,s,\lambda,X)$, we
have the following proposition due to
Davenport~\cite{Davenport1}. 

\begin{proposition}\label{davlem}
  Let $\mathcal R$ be a bounded, semi-algebraic multiset in $\R^n$
  having maximum multiplicity $m$, and that is defined by at most $k$
  polynomial inequalities each having degree at most $\ell$.  Let $\RR'$
  denote the image of $\RR$ under any $($upper or lower$)$ triangular,
  unipotent transformation of $\R^n$.  Then the number of integer
  lattice points $($counted with multiplicity$)$ contained in the
  region $\mathcal R'$ is
\[\Vol(\mathcal R)+ O(\max\{\Vol(\bar{\mathcal R}),1\}),\]
where $\Vol(\bar{\mathcal R})$ denotes the greatest $d$-dimensional 
volume of any projection of $\mathcal R$ onto a coordinate subspace
obtained by equating $n-d$ coordinates to zero, where 
$d$ takes all values from
$1$ to $n-1$.  The implied constant in the second summand depends
only on $n$, $m$, $k$, and $\ell$.
\end{proposition}
Although Davenport states the above lemma only for compact
semi-algebraic sets $\mathcal R\subset\R^n$, his proof adapts without
significant change to the more general case of a bounded semi-algebraic
multiset $\mathcal R\subset\R^n$, with the same estimate applying also to
any image $\mathcal R'$ of $\mathcal R$ under a unipotent triangular
transformation.

\subsubsection{Estimates on reducibility}\label{redsec}

In this subsection, we describe the relative frequencies with which
reducible and irreducible elements sit inside various parts of the
fundamental domain $\FF hL$, as $h$ varies over the compact region $G_0$.

We begin by providing a simple condition which guarantees that  
a point in $V(\Z)$ has discriminant zero.  

\begin{lemma}\label{lem1}
  Let $(A,B)\in V(\Z)$ be an element such that, for some $k\in \{1,\ldots,n/2\}$, the top left $k \times (n-k)$
blocks of entries in $A$ and $B$ are equal to $0$.  Then $\Disc(\Det(Ax-B))=0$.
\end{lemma}

\begin{proof}
If the top left $k \times (n-k)$ blocks of $A$ and $B$ are zero, then $Ax_0-By_0$ (for any $x_0,y_0\in\C$) may be viewed as a block anti-triangular matrix, with square blocks of length $k$, $n-2k$, and $k$ respectively on the anti-diagonal.  Furthermore, the two blocks of length $k$ will be transposes of each other.  There then must be a multiple $Ax_0$ of $A$ and $By_0$ of $B$ such that $Ax_0-By_0$ has the property that one and thus both of these square blocks of length $k$ are singular.  It follows that $(x_0,y_0)$ is a double root of $\Det(Ax-By)$ in $\P^1$, yielding $\Disc(\Det(Ax-By))=0$, as desired. 
\end{proof}



We are now ready to give an estimate on the number of 
elements $(A,B)\in \FF h L \cap V(\Z)$, on average, satisfying $H(A,B)<X$
and $a_{11}=b_{11}=0$:

\begin{proposition}\label{hard}
Let $h$ take a random value in $G_0$ uniformly with respect to the Haar measure
$dg$.  Then the expected number of 
elements $(A,B)\in\FF h L^{(m,\tau)}\cap V(\Z)$ such that 
$H(A,B)< X$  and $a_{11}=b_{11}=0$
is $O_\varepsilon(X^{n+1-\frac1{n}+\varepsilon})$.
\end{proposition}

\begin{proof}
Let $U$ denote the set of all $n(n+1)$ variables
$a_{ij}$ and $b_{ij}$ corresponding to the coordinates on $V(\Z)$.  Each variable $u\in U$ has a {\it weight}, defined as follows. The
action of $s=(s_1,\ldots,s_{n-1})\cdot\lambda$ on $(A,B)\in
V$ causes each variable $u$ to multiply by a certain weight which we
denote by $w(u)$.  These weights $w(u)$ are evidently rational
functions in $\lambda,s_1,\ldots,s_{n-1}$.  
Explicitly, we have
\begin{equation}\label{wtformula}
w(a_{ij})\,=\,w(b_{ij})\,=\,
 \lambda \,\;\prod_{k=1}^{i-1} s_k^{2k}\;\,\prod_{k=i}^{j-1} s_k^{2k-n}\,  \prod_{k=j}^{n-1} s_k^{2k-2n} .
\end{equation}
In particular, 
note that $\prod_{u\in U}w(b)=\lambda^{n(n+1)}$.
For variables $u,u'\in U$, we write $w(u)\leq w(u')$ if $w(u)$ has equal or smaller exponents for 
all parameters $\lambda,s_1,\ldots,s_{n-1}$.  We write $w(u)<w(u')$ if $w(u)\leq w(u')$ and $w(u)\neq w(u')$. 

Let $U_0\subset U$ be any subset of variables in $U$ containing $a_{11}$ and $b_{11}$.  We now give an upper estimate on the total number of $(A,B)\in \FF h L^{(m,\tau)}\cap V(\Z)$ such that all variables in $U_0$ vanish, but all variables in $U\setminus U_0$ do not.  To this end, let 
$V(U_0)$ denote the set of all such $(A,B)\in V(\Z)$.  Furthermore, let $U_1$ denote the 
the set of variables having minimal weights $w(u)$ among the variables $u\in
U\setminus U_0$ (we say that $u$ has ``minimal weight'' in $U\setminus U_0$ if there does not exist any variable $u'\in U\setminus U_0$ satisfying $w(u')< w(u)$).
As explained in~\cite{dodpf}, given $U_1$, it suffices to assume that $U_0$ in fact contains all variables in $u\in U$ such that $w(u)<w(u_1)$ for some $u_1\in U_1$. 

In that case,
if $U_1$ contains any variable $a_{ij}$ or $b_{ij}$ with $i+j>n$, then any element $(A,B)\in V(U_0)$ will be reducible by Lemma~\ref{lem1}.  
Thus, we may assume that all variables $a_{ij}$ and $b_{ij}\in U_1$ satisfy $i+j\leq n$.
In particular, there exist variables $\gamma_1,\ldots,\gamma_{n-1}\in U_1$ such that, for all $i\leq n/2$, 
we have
$$w(\gamma_i)\;\leq \;w(a_{i,n-i})=w(b_{i,n-i})=
\lambda\,\,\prod_{k=1}^{i-1}s_k^{2k}\prod_{k=i}^{n-i-1}\!s_k^{2k-n}\!\prod_{k=n-i}^{n-1}\!s_{k}^{2k-2n}\,.$$
Since $w(a_{i,n-i})=w(b_{i,n-i})=w(a_{n-i,i})=w(b_{n-i,i})$, we may take $\gamma_i=\gamma_{n-i}$ for $i=1,\ldots,n/2$.

For each subset $U_0\subset U$ as above, we wish to show that
$N(V(U_0);X)$, as defined by (\ref{avg}), is $O_\varepsilon(X^{n+1 - 1/n+\varepsilon})$.  
Since $N'(s)$ is absolutely bounded, the equality (\ref{avg}) implies that
\begin{equation}\label{estv0s}
N
(V(U_0);X)\ll
\int_{\lambda=c'}^{X^{1/n}}\!\!\!\int_{s_1,\ldots,s_{n-1}=c}^\infty
\sigma(V(U_0))
\prod_{k=1}^n s_k^{-nk(n-k)}\cdot
d^\times\! s \,d^\times\!\lambda,
\end{equation}
where $c'>0$ is an absolute constant guaranteeing that $E(u,s,\lambda,X)$ has points of height at least 1 and $\sigma(V(U_0))$ denotes the number of integer points
in the region $E(u,s,\lambda,X)$ that satisfy the conditions 
\begin{equation}\label{cond}
\mbox{$\gamma=0$ for $\gamma\in U_0$ and $|\gamma|\geq 1$ for $u\in U_1$}.
\end{equation}

By our construction of $L$, all entries of elements $(A,B)\in G_0L$ are uniformly bounded.
Let $C$ be a constant that bounds the absolute value of all variables
$\gamma\in U_1$ over all elements $\gamma\in G_0L$.
Then, for an element $(A,B)\in E(u,s,\lambda,X)$, we have
\begin{equation}\label{condt}
|\gamma|\leq C{w(\gamma)}
\end{equation}
for all variables $\gamma\in U_1$; therefore, the number of integer points in $E(u,s,\lambda,X)$
satisfying $\gamma=0$ for all $\gamma\in U_0$ and $|\gamma|\geq 1$ for all $\gamma\in U_1$ will be nonzero only if we have
\begin{equation}\label{condt1}
C{w(\gamma)}\geq 1
\end{equation}
for all weights $w(\gamma)$ such that $\gamma\in U_1$.  
Thus if the condition (\ref{condt1}) holds for all weights $w(\gamma)$
corresponding to $\gamma\in U_1$, then---by the definition of $U_1$---we
will also have $Cw(\gamma)\gg 1$ for all weights $w(\gamma)$ such that $\gamma\in
U\setminus U_0$.  In particular, note that we have 
\begin{equation}\label{betaest}
Cw(\gamma_k)\gg 1
\end{equation}
for all $k=1,\ldots,n-1$.

Therefore, if the region $\BB=\{(A,B)\in
E(u,s,\lambda,X):\gamma=0\;\;\forall \gamma\in U_0;\;\; |\gamma|\geq 1\;\; \forall
\gamma\in U_1\}$ contains an integer point, then (\ref{condt1}) and
Proposition~\ref{davlem} together imply that the number of integer points
in $\BB$ is $O(\Vol(\BB))$, since the volumes of all the projections of
$u^{-1}\BB$ will in that case also be $O(\Vol(\BB))$.  
Now clearly
\[\Vol(\BB)=O\Bigl(\prod_{\gamma\in U\setminus U_0} w(\gamma)\Bigr),\]
so we obtain
\begin{equation}\label{estv1s}
N(V(U_0);X)\ll
\int_{\lambda=c'}^{X^{1/n}} \!\!\!\int_{s_1,\ldots,s_{n-1}=c}^\infty
\prod_{\gamma\in U\setminus U_0} w(\gamma)
\, \prod_{k=1}^n s_k^{-nk(n-k)}
\cdot d^\times\! s \,d^\times\!\lambda.
\end{equation}

We wish to show that the latter integral is bounded by $O_\varepsilon(X^{n+1 - {1}/{n}+\varepsilon})$ for every choice of $U_0$.  If, for example, the total exponent of
$s_k$ in (\ref{estv1s}) is nonpositive for all $k$ in $\{1,\ldots,n-1\}$,
then it is clear that the resulting integral will be at most
$O_\varepsilon(X^{(n(n+1)-|U_0|)/n+\varepsilon})$ in value, since each $s_k$ is bounded above by a power of $X$ by (\ref{betaest}).  We will show below that this condition holds, e.g., when $a_{1,j}$ and $b_{1,j}$ are both in $U_1$ for some $j\leq n/2$.  

For cases where this nonpositive exponent condition does not hold, we observe that, due to~(\ref{betaest}), the integrand in
(\ref{estv1s}) may be multiplied by the absolute value $\pi$ of any product of the variables $\gamma_k$ ($k=1,\ldots,n-1$) without harm, and the
estimate (\ref{estv1s}) will remain true.  
Extend the notation $w$ multiplicatively, i.e., $w(\gamma\delta)=w(\gamma)w(\delta)$.
Then it suffices to choose an appropriate weighted product $\pi$ of variables $\gamma_1,\ldots,\gamma_{n/2}$ in $U_1$ such that: (1) 
$\#\pi<\#U_0$ (where $\#\pi$ denotes the
total number of variables of $U$ appearing in $\pi$, counted with
multiplicity), and (2)
$w(\pi)w(\prod_{\gamma\in U\setminus U_0}\gamma)\prod_{k=1}^{n-1}s_k^{-nk(n-k)}$ has nonpositive exponent of $s_i$ for all $i\in\{1,\ldots,n-1\}$.  Indeed, if we have such a $\pi$, then
applying the
inequalities (\ref{condt}) to each of the variables in $\pi$ yields
\begin{equation}\label{estv2s}
N(V(U_0);X)\ll
\int_{\lambda=c'}^{X^{1/n}} \!\!\!\int_{s_1,\ldots,s_{n-1}=c}^\infty
\prod_{\gamma\in U\setminus U_0} w(\gamma)\;w(\pi)
\, \prod_{k=1}^n s_k^{-nk(n-k)}
\cdot d^\times\! s \,d^\times\!\lambda.
\end{equation}
By our assumptions on $\pi$, this then gives the desired estimate
$$N(V(U_0);X)=O_\varepsilon(X^{(n(n+1)-\#U_0+\#\pi)/n+\varepsilon})=
O_\varepsilon(X^{n+1 - {1}/{n}+\varepsilon}).$$

We now choose such a $\pi$ for each candidate $U_0\subset U$.
As mentioned earlier, we claim that we may simply take $\pi=1$ if $a_{1j}$ and $b_{1j}$ are both in $U_1$ for some $j\leq n/2$.  Indeed, in such a case, $U_0$ is contained in the set $U'$ of all variables $a_{ij},b_{ij}$ such that $j<n/2$.  We then note that for each $k$, the sum of the exponents of $s_k$ in $w(\gamma)$, over all $\gamma\in U'$ where the exponent of $s_k$ in $w(\gamma)$ is negative, is equal~to 
$$\!\left\{\!\!\!\!\begin{array}{lcll}
-2\Bigl((2n-2k){\displaystyle{\frac{k(k+1)}{2}}}+\bigl(\frac{n}{2}-k-1\bigr)k(n-2k)\Bigr) &\!\!\!\!=\!\!\!\!& -nk(n-k) + 2k^2(\frac{n}2-k-1) & \mbox{\!\!\!if $1\leq k<
\frac{n}{2}$;}\\[.1in]
-2\Bigl((2n-2k)\displaystyle{\frac{\frac n2(\frac n2-1)}{2}}\Bigr) &\!\!\!\!=\!\!\!\!& -nk(n-k) + \frac n2(k+1-\frac{n}2)(n-k) & \mbox{\!\!\!if $\frac n 2\leq k<n$.}\\
\end{array}\right.
$$
It follows that, for any $U_0\subset U'$, we have
$$
\prod_{\gamma\in U\setminus U_0}\!w(\gamma)\,\prod_{k=1}^{n-1}s_k^{-nk(n-k)}$$ has nonpositive exponent of $s_i$ for all $i\in\{1,\ldots,n-1\}$. 
Hence we may take $\pi=1$ for all $U_0\subset U'$, as claimed. 

It remains to consider the situations where at least one of $a_{1,n/2}$ and $b_{1,n/2}$ is in $U_0$.  In such a case, we take 
$$ \pi = \prod_{a_{ij}\in U_0} |\gamma_{i}\gamma_{j}|^{1/2}
\prod_{b_{ij}\in U_0} |\gamma_{i}\gamma_{j}|^{1/2}\prod_{k=1}^{n/2} \,\,\gamma_k^{-2/n}.$$
The assumption that at least one of $a_{1,n/2}$ or $b_{1,n/2}$ is in $U_0$ guarantees that every $\gamma_i$ occurs in $\pi$ to a nonnegative power.  Since one easily checks that
$$\prod_{k=1}^{n/2} \,\,\gamma_k^{-2/n} = \lambda\prod_{k=n/2}^{n-1}s_k^{2},$$
we have
\begin{equation}\label{mustbound}
\begin{array}{l}
  \displaystyle w(\pi)\prod_{\gamma\in U\setminus U_0}\!w(\gamma) \, \prod_{k=1}^n s_k^{-nk(n-k)} \\ \qquad\qquad =\displaystyle 
\lambda\prod_{a_{ij}\in U_0}w(a_{ij})^{-1}w(\gamma_i\gamma_j)^{1/2}
\prod_{b_{ij}\in U_0}w(b_{ij})^{-1}w(\gamma_i\gamma_j)^{1/2}
\prod_{k=n/2}^{n-1}s_k^{2}
\, \prod_{k=1}^n s_k^{-nk(n-k)},
\end{array}
\end{equation}
and we wish to check that the exponent of $s_i$ in (\ref{mustbound}) is nonpositive for all $i$. 
Now the exponent of $s_i$ in $w(a_{ij})^{-1}w(\gamma_i\gamma_j)^{1/2}$ (or $\prod_{b_{ij}\in U_0}w(b_{ij})^{-1}w(\gamma_i\gamma_j)^{1/2}$) is nonnegative for
all $i$ by construction. It follows that, for every $i$, the maximum value of the exponent of $s_i$ in
(\ref{mustbound}) is achieved 
when $U_0=U^-$, where $U^-$ consists of all variables $a_{ij}$ and $b_{ij}$ where $i+j<n$.  In this case, one easily checks that the exponent of $s_i$ in (\ref{mustbound}) is equal to
$$ \prod_{i=1}^{n/2-1}s_i^0\prod_{i=n/2}^{n-1}s_i^{2-n(2i-n+1)},$$
which has nonpositive exponent of $s_i$ for all $i\in\{1,\ldots,n-1\}$. 
This completes the proof of Proposition~\ref{hard}.
\end{proof}

We now give an estimate on the number of 
reducible elements $(A,B)\in \FF h L \cap V(\Z)$, on average, satisfying
$a_{11}\neq 0$ or $b_{11}\neq 0$:

\begin{proposition}\label{hard2}
Let $h\in G_0$ be any element.
Then the 
number of 
reducible elements $(A,B)\in\FF h L^{(m,\tau)}\cap V(\Z)$ such that 
$H(A,B)< X$  and $a_{11}\neq 0$ or $b_{11}\neq 0$
is $o(X^{n+1})$.
\end{proposition}
Proposition~\ref{hard2} implies the claim made after Theorem~\ref{countvz}, namely, that Theorem~\ref{countvz} remains true even if ``nondegenerate'' is replaced by ``irreducible''.  We defer the proof of Proposition~\ref{hard2} to the end of the section.

By Theorem~\ref{woodsymfield}, the stabilizer in $G(\Q)$ of an element $v\in V(\Q)$ has size greater than the minimum possible size 2 precisely when the invariant binary $n$-ic form $f_v$  is reducible. Proposition~\ref{hard2} thus also implies the following result which bounds the number of
$G(\Z)$-equivalence classes of integral elements in $\FF hL^{(m,\tau)}$ with height less than $X$ that have 
large stabilizers inside $G(\Q)$:

\begin{proposition}\label{gzbigstab}
  Let $h\in G_0$ be any element.
  Then the number of 
  elements $(A,B)\in \FF h L^{(m,\tau)}\cap V(\Z)
  $ 
  such that $H(A,B)< X$
  whose stabilizer in $G(\Q)$ has size greater
  than $2$ is $o(X^{n+1})$.
\end{proposition}

\subsection{The main term}

Fix again $m,\tau$ and let $L=L^{(m,\tau)}$.  The work of the previous subsection shows that, in order to obtain Theorem~\ref{thmcount}, it suffices to count those integral elements $(A,B)\in \FF h L$ of bounded height for which $a_{11}\neq 0$ or $b_{11}\neq 0$, as $h$ ranges over $G_0$.

Let $\RR_X(hL)$ denote the region $\FF h L\cap \{(A,B)\in
V(\R):H(A,B)<X\}$.  Then we have the following result counting the
number of integral points in $\RR_X(hL)$, on average, satisfying
$a_{11}\neq 0$ or $b_{11}\neq 0$:

\begin{proposition}\label{nonzerob11}
  Let $h$ take a random value in $G_0$ uniformly with
  respect to the Haar measure $dg$.  Then the expected
  number of elements $(A,B)\in \FF hL\cap V(\Z)$ such that
  $|H(A,B)|< X$ and $a_{11}\neq 0$ or $b_{11}\neq0$ is $\Vol(\RR_X(L))
  + O(X^{n+1-{1}/{n}})$.
\end{proposition}

\begin{proof}
Following the proof of Proposition~\ref{hard}, let $V^{(m,\tau)}(\!\varnothing\!)$ denote the subset of $V(\R)$ such that $a_{11}\neq0$.  
We wish to show that
\begin{equation}\label{toprove2}
N
(V^{(m,\tau)}(\!\varnothing\!);X)=\Vol(\RR_X(L)) + O(X^{n+1-{1}/{n}}).
\end{equation}
We have
\begin{equation}\label{translate2}
  N
  (V^{(m,\tau)}(\!\varnothing\!);X)=\frac{1}{C^{(m,\tau)}_{G_0}}\int_{\lambda=c'}^{X^{1/n}} \!\!\!
\int_{s_1,\ldots,s_{n-1}=c}^\infty
\int_{u\in    N'(s)} 
  \sigma(V(\!\varnothing\!))
  \prod_{k=1}^n s_k^{-nk(n-k)}
  \cdot du\,d^\times\! s \,d^\times\!\lambda,
\end{equation}
where $\sigma(V(\!\varnothing\!))$ denotes the number of integer points
in the region $E(u,s,\lambda,X)$ satisfying $|a_{11}|\geq 1$.
Evidently, the number of integer points in $E(u,s,\lambda,X)$
with $|a_{11}|\geq 1$
can be nonzero only if we have
\begin{equation}\label{condt2}
C{w(a_{11})}=C\cdot\frac{\lambda}
{{\prod_{k=1}^{n-1} s_k^{2n-2k}}}
\geq 1.
\end{equation}
Therefore, if the region $\BB=\{(A,B)\in
E(u,s,\lambda,X):|a_{11}|\geq 1\}$
 contains an integer point, then (\ref{condt2}) and
Proposition~\ref{davlem} imply that the number of integer points in $\BB$
is $\Vol(\BB)+O(\Vol(\BB)/w(a_{11}))$, since all
smaller-dimensional projections of $u^{-1}\BB$ are clearly bounded by
a constant times the projection of $\BB$ onto the hyperplane $a_{11}=0$ (since
$a_{11}$ has minimal weight).

Therefore, since $\BB=E(u,s,\lambda,X)-\bigl(E(u,s,\lambda,X)-\BB\bigr)$, 
we may write
\begin{eqnarray}\label{bigint}\nonumber
\!\!N
(V^{(m,\tau)}(\!\varnothing\!);X) &\!\!\!\!\!\!\!=\!\!\!\!\!\!& \!\frac1{C^{(m,\tau)}_{G_0}}
\!\!\int_{\lambda=c'}^{X^{1/n}}\!\!\!\!\!
\int_{s_1,\ldots,s_{n-1}=c}^{\infty} \int_{u\in N'(s)}\!\!
\Bigl(\Vol\bigl(E(u,s,\lambda,X)\bigr)\!-\!\Vol\bigl(E(u,s,\lambda,X)\!-\!\BB\bigr) 
 \\[.085in] & & \,\,\,\,\,\,\,\,\,\,
\,\,\,\,\,\,\,+O(\max\{\lambda^{n(n+1)-1}
{\prod_{k=1}^{n-1} s_k^{2k-2n}},
1\})
\Bigr)   
\, \prod_{k=1}^n s_k^{-nk(n-k)}\cdot du\,
d^\times s\, d^\times \lambda.
\end{eqnarray}
The integral of the first term in (\ref{bigint}) is $\int_{h\in G_0}
\Vol(\RR_X(hL))dh$.
Since $\Vol(\RR_X(hL))$ does not
depend on the choice of $h\in G_0$,
the latter integral is simply $C^{(m,\tau)}_{G_0}\cdot \Vol(\RR_X(L))$.

To estimate the integral of the second term in (\ref{bigint}), let $\BB'=
E(u,s,t,X)-\BB$, and for each $|a_{11}|< 1$, let $\BB'(a_{11})$ be
the subset of all elements $B\in\BB'$ with the given value of
$a_{11}$.  Then the $(n(n+1)-1)$-dimensional volume of $\BB'(a_{11})$ is at most 
$O\Bigl(\prod_{u\in U\setminus\{a_{11}\}}w(b)\Bigr)$, and so we
have the estimate 
\[\Vol(\BB') \ll \int_{-1}^1 \prod_{u\in
  U\setminus\{a_{11}\}}w(u)
\,\,da_{11} = O\Bigl(\prod_{u\in U\setminus\{a_{11}\}}w(u)\Bigr).\]
The second term of the integrand in (\ref{bigint}) can thus be
absorbed into the third term.
 
Finally, one easily computes the
integral of the third term in (\ref{bigint}) to be
$O(X^{n+1-1/n})$.  We thus obtain
\begin{equation}\label{obtain}
N
(V^{(m,\tau)}(\!\varnothing\!);X) = 
\Vol(\RR_X(L)) 
+ O(X^{n+1-1/n}).
\end{equation}
As the identical argument applies when $b_{11}\neq 0$, this completes the proof.
\end{proof}

Propositions~\ref{hard}, \ref{hard2}, \ref{gzbigstab}, and \ref{nonzerob11} now yield the first assertion of Theorem~\ref{thmcount}. To obtain the second assertion, we must compute 
the volume ${\rm Vol}(\mathcal R_X(L))$.

\subsection{Computation of the volume}\label{secvol}

Fix again $m,\tau$.  In this subsection, we describe how to compute
the volume of $\mathcal R_X(L^{(m,\tau)})$.  

To this end, let $R^{(m,\tau)}:=\Lambda L^{(m,\tau)}$.
For each $(f_0,\ldots,f_{n})\in \R^{n+1}$, the set
$R^{(m,\tau)}$ contains at most one point
$p^{(m,\tau)}(f_0,\ldots,f_{n})$ having invariants
$f_0,\ldots,f_{n}$. Let $R^{(m,\tau)}(X)$ denote the set of all those
points in $R^{(m,\tau)}$ having height less than $X$. Then
$\Vol(\mathcal R_X(L^{(m,\tau)}))=
\Vol(\FF_1
\cdot
R^{(m,\tau)}(X))$, where 
$\FF_1$ denotes the fundamental domain
$N'A'K$ for the action of 
$G(\Z)$ on $G(\R)$ 
(here $N'$, $A'$,
and $K$ are as defined in the paragraph before (\ref{nak})).

The set $R^{(m,\tau)}$ is in canonical one-to-one correspondence with the
set $\{(f_0,\ldots,f_n)\in I(m)\}$, where $I(m)$ again denotes
the component of the space of binary $n$-ic forms in $\R^{n+1}$ having
nonzero discriminant and $2m$ real roots.
 There is thus a
natural measure on each of these sets $R^{(m,\tau)}$, given by
$dr=df_0\cdots df_{n}$.  
Let $\omega$ be a differential which generates the rank
$1$ module of top-degree differentials of 
$G$ over $\Z$.
We begin with the following key proposition, which describes how one
can change measure from $dv$ on $V$ to $\omega(g)\,dr$ on $G\times R$:
\begin{proposition}\label{genjac}
  Let $K$ be $\C$, $\R$, or $\Z_p$ for any prime $p$. Let $dv$ be the
  standard additive measure on $V(K)$. Let $R$ be an open subset of
  $K^{n+1}$ and let $s:R\to V(K)$ be a continuous function such that
  the invariants of $s(f_0,\ldots,f_{n})$ are given by
  $f_0,\ldots,f_{n}$.  Then there exists a rational nonzero
  constant~$\mathcal J$ $($independent of $K,$ $R,$ and $s)$ such
  that, \,for any measurable function $\phi$ on $V(K)$, we have
  \begin{equation}\label{eqjacpgl2}
\int_{v \in G(K)\cdot
  s(R)}\phi(v)dv\,=\,|\mathcal J|\int_{R}\int_{G(K)}
\phi(g\cdot s(f_0,\ldots,f_n))\,\omega(g) \,dr
  \end{equation}
where we regard $G(K)\cdot s(R)$ as a
multiset.
\end{proposition}
The proof of Proposition~\ref{genjac} is identical to that of
\cite[Prop.~3.11]{BS} (where we use the important fact that the
dimension $n(n+1)$ of $V$ is equal to the sum of the degrees of the
invariants $f_0,\ldots,f_{n}$ for the action of $G$ on $V$).

Proposition \ref{genjac} may now be used
to give a convenient expression for the volume of the multiset~$\mathcal R_X(L^{(m,\tau)})$:
\begin{eqnarray}\nonumber
\!\!\!\!\int_{\mathcal R_X(L^{(m,\tau)})}\!\!\!\!\!dv=\int_{\FF_1\cdot R^{(m,\tau)}(X)}\!\!\!\!\!dv&=&
|\mathcal J|\cdot\int_{R^{(m,\tau)}(X)}\int_{\FF_1}dg\,dr \\ \nonumber &=&|\mathcal J|\cdot \Vol(G(\Z)\backslash G(\R))\cdot \int_{R^{(m,\tau)}(X)}dr \\ \label{volexp}&=& |\mathcal J|\cdot 
\zeta(2)\zeta(3)\cdots \zeta(n) \nonumber
\cdot \Vol\bigl(\{f\in I(m)\subset \R^{n+1}:H(f)<X\}\bigr).\\ \label{volformula}
\end{eqnarray}

We now show that $|\mathcal J|= 2^{n}$.  It suffices to evaluate $|\mathcal J_p|$ for each $p$.  Fix a binary $n$-form $\phi$ over $\F_p$ that is irreducible and has leading coefficient 1.  Let $S\subset V(\Z_p)$ be the set of elements whose invariant binary $n$-form reduces to $\phi$ modulo $p$. Then 
\begin{equation}\label{volS1}
\begin{array}{rcl}
    \Vol(S)&=&|\J|_p
\Vol(G(\Z_p))\displaystyle\int_{f\equiv \phi\!\!\!\!\! \pmod{p}}
\Bigl(\displaystyle
\sum_{\textstyle\frac{\{v\in V_{\Z_p}:f_v\equiv \phi\!\!\!\!\!\pmod{p} \}}{G(\Z_p)}}
\frac{1}{\#\Aut_{\Z_p}(v)}
\Bigr)\,dr\\[.5in]
& = & \displaystyle|\J|_p \Vol(G(\Z_p)) \cdot\frac{1}{p^{n+1}}\cdot\frac{|(\Z_{p^n}^{\times}/\Z_{p^n}^{\times2})_{N\equiv 1}|}{\Z_{p^n}^\times[2]}, 
\end{array}
\end{equation}
where $\Aut_{\Z_p}(v)$ denotes the stabilizer of $v$ in $G(\Z_p)$ and $\Z_{p^n}$ denotes the unique local ring of rank $n$ over $\Z_p$ having residue field
$\F_p$.
On the other hand, since $S$ is defined by congruence conditions modulo $p$,
by the results of \S\ref{fporbits}, we have
\begin{equation}\label{volS2}
\Vol(S)=|\SL_n(\F_p)|/p^{n(n+1)}.
\end{equation}
Equating the two expressions for $\Vol(S)$, we obtain
\begin{equation} \label{jvalue}
|\J|_p = \left\{\begin{array}{cl}
1/2^{n}& \mbox{if $p=2$; and} \\
1 & \mbox{otherwise.}
\end{array}\right.
\end{equation}
Therefore, $\J=2^{n}$.

We have proven Theorem~\ref{thmcount}.

\subsection{Congruence conditions}\label{secbqcong}

In this subsection, we prove the following version of Theorem \ref{thmcount} where we
count elements of $V(\Z)$ satisfying some finite set of congruence conditions:

\begin{theorem}\label{cong2}
Suppose $S$ is a subset of $V(\Z)$ defined by 
congruence conditions modulo finitely many prime powers. Then 
\begin{equation}\label{ramanujan}
N(S\cap V^{(m,\tau)};X)
  \;=\; N(V(\Z)^{(m,\tau)};X)\cdot
  \prod_{p} \mu_p(S)+o(X^{n+1}),
\end{equation}
where $\mu_p(S)$ denotes the $p$-adic density of $S$ in $V(\Z)$.
\end{theorem}

To obtain Theorem~\ref{cong2}, suppose $S$ is defined by congruence conditions modulo some integer~$m$. Then $S$ may be viewed as the union of (say) $k$ translates ${\mathcal L}_1,\ldots,{\mathcal L}_k$ of the lattice $m\cdot V(\Z)
$. For each such lattice translate ${\mathcal L}_j$, we may use formula (\ref{avg}) and the discussion following that formula to compute $N(S;X)$, but where each $d$-dimensional volume is scaled by a factor of $1/m^d$ to reflect the fact that our new lattice has been scaled by a factor of $m$. For a fixed value of $m$, we thus obtain
\begin{equation}\label{lat}
N({\mathcal L}_j\cap V^{(m,\tau)};X) = m^{-n(n+1)} \Vol(R_X(L)) +o(X^{n+1}).
\end{equation}
Summing (\ref{lat}) over $j$, and noting that $km^{-n(n+1)} = \prod_p \mu_p(S)$, yields Theorem~\ref{cong2}.

\subsection{Proof of Proposition~\ref{hard2}}

We may use the results of the previous subsection to prove Proposition~\ref{hard2}.
First, note that if an element $(A,B)\in V(\Z)$ 
is reducible over $\Q$ then it also must be reducible modulo $p$ 
(i.e., its image in $V(\F_p)$ has invariant binary $n$-ic form that is reducible over $\F_p$)
for every $p$.

Let $S^\red$ denote the set of elements in $V(\Z)$ that are reducible over $\Q$, and let $S^\red_p$ denote the set of all elements in $V(\Z)$ that are reducible modulo $p$.  Then $S^\red\subset \cap_p S^\red_p$.  Let $S^\red(Y)=\cap_{p<Y}S^\red_p$, and let us use $V^{(m,\tau)}(\!\varnothing\!)$ to denote the set of all $(A,B)\in V(\Z)^{(m,\tau)}$ such that $a_{11}\neq0$ or $b_{11}\neq 0$.  Then the proof of Theorem~\ref{cong2} (without assuming Proposition~\ref{hard2}!) gives
\begin{equation}\label{SYcount}
N
(S^\red(Y)\cap V^{(m,\tau)}(\!\varnothing\!);X)
  \;\leq\; N
  (V^{(m,\tau)}(\!\varnothing\!);X)\cdot
  \prod_{p<Y} \mu_p(S^\red_p)+o(X^{n+1}).
\end{equation}
To estimate $\mu_p(S^\red_p)$, we recall from \S\ref{fporbits} that the number of elements $(A,B)\in V(\F_p)$ having any given separable invariant binary $n$-ic form $f(x,y)=\det(Ax-By)$ is $\#\SL_n(\F_p)$.  It is an elementary and well-known calculation that the number of separable binary $n$-ic forms $f(x,y)$
over $\F_p$ that are irreducible over $\F_p$ is $p^{n+1}/n+O(p^{n})$, where the implied $O$-constant is independent of $p$.  
It follows that
$$
\mu_p(S^\red_p) \leq 1-\frac{p^{n+1}}{n}\cdot\frac{\SL_n(\F_p)}{p^{n(n+1)}} + O(1/p)=
\frac{n-1}n+O(1/p).
$$
Combining with (\ref{SYcount}), we see that
$$\lim_{X\to\infty}\frac{N
(S^\red\cap V^{(m,\tau)}(\!\varnothing\!);X)}{X^{n+1}}
 \;\ll\;  \prod_{p<Y} \mu_p(S^\red_p) \;\ll \; \prod_{p<Y}\Bigl(\frac{n-1}n+O(1/p)\Bigr).
 $$
 When $Y$ tends to infinity, the product on the right tends to 0, proving Proposition~\ref{hard2}.

\section{Proofs of the main theorems}\label{mainproof}

To prove Theorem 1, 
let $\mu(I(m))$ denote $\Vol(\{f\in I(m)\subset \R^{n+1}:H(f)\leq1/2\})$,
so that $\mu(I(m))$ represents the probability that a random
polynomial with i.i.d.\ coefficients in the unit interval $[-1/2,1/2]$
has $2m$ real roots.  Let $\mu(I_p(m))$ denote $|I_p(m)|/p^{n+1}$, so
that $\mu(I_p(m))$ denotes the probability that a random degree $n$
polynomial over $\F_p$ has $m$ distinct irreducible factors.  Similarly, define
$\mu(I_8(m))$ to be $|I_8(m)|/8^{n+1}$, so  $\mu(I_p(m))$ denotes the probability that a random degree $n$
polynomial over $\Z/8\Z$ has a factorization with $m$ distinct irreducible factors modulo~2.
Finally,
note that if $(x_0,y_0,z_0)$ is an integral solution to $z^2=f(x,y)$,
where $x_0,y_0$ are relatively prime integers, 
then $(-x_0,-y_0,z_0)$ is also a solution, and these two solutions give rise to distinct
integral orbits for 100\% of binary $n$-ic forms $f$ via the
construction of Section~2 (since $-1$ will not be a square in $K_f$ for 100\% of $f$). 
Using the fact that
$$\zeta(2)\zeta(3)\cdots\zeta(n)\prod_p\frac{\#\SL_n(\F_p)}{p^{n^2-1}} = 1,$$
we conclude from Theorems~\ref{thmcount}
and \ref{cong2} and the results of \S\ref{secdensity} that the density of hyperelliptic curves over $\Q$ of
genus $g$ having a rational point is at most
\begin{equation}
\frac{1}2 \cdot 2^{n/2}
\sum_{m=0}^{n/2}\frac{\max\{1,2m\}}{2^{m-1}}
\mu(I(m))\cdot \frac{1}{2^n}
\sum_{m=0}^n
\frac{12}{2^{m-1}}
\mu(I_8(m))\cdot
\prod_p\sum_{m=0}^n\sum_{f\in I_p(m)}\!\!
\min\Bigl\{1,\frac{p+1}{2^{m-1}}\Bigr\}
\mu(I_p(m))
\label{finaldensity} 
\end{equation}
and 
this tends to 0 rapidly as $n\to\infty$.
Indeed, the archimedean factor is $O(2^{n/2})$, while the factor at 2 is $O(2^{-n})$, and this immediately gives that (\ref{finaldensity}) is at most $O(2^{-n/2})$, or $O(2^{-g})$.

We can improve this estimate by noting that even after removing $2^{n/2}$ and $2^{-n}$ from the factors at infinity and 2, respectively, the sums that appear for each prime in (\ref{finaldensity}) individually tend to 0 as $n$ tends to infinity.  
Indeed,  the main results of 
\cite{DPSZ}, which show in particular that the density of real polynomials having fewer than $\log n/\log\log n$ real zeroes is $O(n^{-b+o(1)})$ for some $b>0$, immediately imply that the archimedean sum above is
bounded by $O(n^{-\delta/\log \log n})$ for some $\delta>0$.
Next, we observe that the sum at the 
prime $p$ can be bounded by $O(p/\sqrt{n})$ (provided $p$ is at most a
small fixed power of $n$) using the results of \cite{Car} on the distribution of the number of factors of polynomials over finite fields.  Multiplying this bound
over all primes $p$ less than a small power of $n$ is then enough to
bound the product of the sums in (\ref{finaldensity}) by $O(e^{-\delta_1n^{\delta_2}})$ for some
positive constants $\delta_1$ and $\delta_2$.  We conclude in particular that (\ref{finaldensity}) is $o(2^{-g})$.
(More precisely, it is $O(2^{-g-\delta_1 g^{\delta_2}})$ for some $\delta_1,\delta_2>0$.) It follows that the lower density $\rho_g$ of hyperelliptic curves over $\Q$ of genus $g$, when ordered by height, having no rational points is at least $1-o(2^{-g})$.  We have proven Theorem~\ref{main} and the remarks following it.

\vspace{.1in}
We may refine the estimate (\ref{finaldensity}) further by working over $\Q_p$ rather than $\F_p$, following the methods of \cite{BS}.  Although this does not significantly improve the estimates we have already described above as $g\to\infty$ for Theorem~\ref{main} and Corollary~\ref{hasse}, such a refinement can be useful conceptually and also for understanding individual genera. We sketch the method here.  To carry this out, we first observe that the results of 
\cite{AITII} 
imply that two rational points $P$ and $P'$ on the hyperelliptic curve $C:z^2=f(x,y)$ map to common $\SL_n(\Q)$-orbits in $\Q^2\otimes\Sym_2(\Q^n)$ via the construction of Section 2 if and only if they differ by an element of $2\Jac(C)(\Q)$.  The stabilizer in $\SL_n(\Q)$ of elements in such orbits are closely related to $\Jac(C)[2](\Q)$. 
By assigning local weights to the integral orbits appropriately, as in \cite{BS}, so that each $\Q_p$-rational point on $C$, up to translation in $2\Jac(C)(\Q_p)$, gets counted with a weight of $1/\Jac(C)(\Q_p)[2]$, we may then improve the estimate  (\ref{finaldensity}) to
 \begin{equation}\label{finaldensity2}
\displaystyle\int_{\substack{C=C(f_0,\ldots,f_{n})/\R\\H(C)\leq1/2}}\!\!\!\!
\frac{\#(C(\R)/2\Jac(C)(\R))}
{\#\Jac(C)[2](\R)}df_0\cdots df_{n}\,
\displaystyle{\prod_p
\int_{C=C(f_0,\ldots,f_{n})/\Z_p}\!\!\!\!\!\!
\frac{\#({C(\Q_p)}/{2\Jac(C)(\Q_p)})}
{\#\Jac(C)[2](\Q_p)}df_0\cdots df_{n}};
\end{equation}
here, by ``${C(\Q_\nu)}/{2\Jac(C)(\Q_\nu)}$'', we mean the set $C(\Q_\nu)$ modulo equivalence, where two points of $C(\Q_\nu)$ are called equivalent if they differ by an element of $2\Jac(C)(\Q_\nu)$.
In fact, analogous to the calculations in \cite{BS}, this expression gives an upper bound on the density of hyperelliptic curves of genus $g$ over~$\Q$ having rational points\, among those hyperelliptic curves that are {\it locally soluble}, i.e., locally have a point at every place. (Of course, all other non-locally-soluble hyperelliptic curves automatically have no rational point!).  
Again, expression (\ref{finaldensity}) gives a convenient way to estimate the more precise (\ref{finaldensity2}) in practice, although we will have occasion to use (\ref{finaldensity2}) directly as well, e.g., when analyzing the case of genus one in more detail in the next section.  

\vspace{.1in} Let us turn next to the consequences of the above estimates for the fake 2-Selmer
sets of hyperelliptic curves. For a field $k$ and a hyperelliptic
curve $C:z^2=f(x,y)$ over $k$, let us use $\xi_k$ to denote the map
taking $k$-solutions of $z^2=f(x,y)$ to $(K_f^\times/K_f^{\times2})_{N\equiv f_0}$, as defined
in Sections~2 and 3.  Then when reduced modulo $k^\times$, this yields
the famous ``$x-T$'' map of Cassels~\cite{Cassels2}, which is written as $\mu_k:C(k)\to H(k)$ in \cite[\S2]{BruinStoll}, where $H_k:=(K_f^\times/K_f^{\times2}k^\times)_{N\equiv f_0}$. For a
number field $k$ and a place $\nu$ of $k$, let $\rho_\nu:H_k\to
H_{k_\nu}$ denote the natural map.  Then Bruin and Stoll~\cite{BruinStoll}
define the {\it fake 2-Selmer set} of $C/k$ by
$$\Sel^{(2)}_{\rm fake}(C/k):=\{\alpha\in H_k:\rho_\nu(\alpha)\in \mu_{k_\nu}(C(k_{\nu}))\mbox{ for all
places $\nu$ of $k$}\}.$$
In particular, $\Sel^{(2)}_{\rm fake}(C/k)=\varnothing$ implies that $C$ has no $k$-rational points.

We may ask what is the average size of $\sf$ over all locally soluble hyperelliptic
curves $C$ over $\Q$ of genus $g$, when ordered by height.  We claim
that (\ref{finaldensity}) or (\ref{finaldensity2}) also gives an
upper bound on the (limsup of the) average size of $\sf$.  To
prove this, it suffices to demonstrate that every element of the fake
2-Selmer set $\sf$ of $C:z^2=f(x,y)$ over $\Q$, when viewed as a
union~$U$ of $G(\Q)$-orbits on $V(\Q)$ having invariant form $f$ via
Theorem~\ref{woodsymfield}, has a representative in $V(\Z)$.  The
local conditions that we have imposed on $V(\Z)$ to sieve to those
orbits in $V(\Z)$ that arise locally from local points on~$C$ are then the
same conditions that we would impose to sieve to those orbits in
$V(\Z)$ corresponding to elements of the fake 2-Selmer set of $C$.

We now demonstrate that every element $[\alpha]\in \sf$ for a hyperelliptic curve $C$ has an associated integral point in $V(\Z)$
contained in the union $U$ of $G(\Q)$-orbits on $V(\Q)$ associated to $[\alpha]=\{r\alpha\}_{r\in \Q^\times}$ via Theorem~\ref{woodsymfield}.  To do this, given $\alpha\in(K_f^\times/K_f^{\times2})_{N\equiv f_0}$ such that $[\alpha]\in\sf$, we show that there exists $r\alpha\in (K_f^\times/K_f^{\times2})_{N\equiv f_0}$ with $r\in\Q^\times$ whose associated $G(\Q)$-orbit on $V(\Q)$ has a representative in $V(\Z)$.  

We first accomplish this locally at each finite prime $p$ of $\Q$.
Let $C:z^2=f(x,y)$ be a hyperelliptic curve $C$ where $f$ has
coefficients in $\Z_p$.  Given any $\Q_p$-point on $C$, the set of
$G(\Q_p)$-orbits on $V(\Q_p)$ associated to this rational point
corresponds to a class $[\alpha]\in H_{\Q_p}$, by the arguments at the
end of \S\ref{genfield}.  Similarly, the set of $G(\Z_p)$-orbits on $V(\Z_p)$
associated to this rational point corresponds to a set of pairs
$(I,\alpha)$ via Theorem~\ref{woodsym}, where the set of possible 
$\alpha\in(K_f^\times/K_f^{\times2})_{N\equiv f_0} $ is again closed
under multiplication by elements in $\Z_p^\times$ (by the same
argument at the end of \S\ref{genfield}).  It follows that, given
$\alpha\in(K_f^\times/K_f^{\times2})_{N\equiv f_0}$ in its class
$[\alpha]\in H_{\Q_p}$, either the $G(\Q_p)$-orbit on $V(\Q_p)$
corresponding to $\alpha$ or that corresponding to $p\alpha$ (or both)
will contain a point of $V(\Z_p)$.  (This is because
$\Q_p^\times/\Q_p^{\times2}\Z_p^\times$ is generated by $p$.)

Returning to the global case, let us fix any $\alpha\in
(K_f^\times/K_f^{\times2})_{N\equiv f_0}$ whose class $[\alpha]$ lies
in $\sf$.  Let $v$ be any element in the associated $G(\Q)$-orbit on
$V(\Q)$ having binary invariant form $f$.  Then $v$ is already in
$V(\Z_p)$ for all but finitely many primes $p$.  For the remaining
primes $p$, the $G(\Q_p)$-orbit on $V(\Q_p)$ corresponding to 
$\alpha$ or $p\alpha$ will contain a
point of $V(\Z_p)$.  Let $r$ denote the product of those primes $p$
where an additional factor of $p$ is required for an integral
$\Z_p$-orbit.  Then the $G(\Q_p)$-orbit on $V(\Q_p)$
corresponding to the class $r\alpha$ contains a point of $V(\Z_p)$ for
every prime~$p$.  Therefore, an element $v'\in V(\Q)$ corresponding to the class
$r\alpha\in
(K_f^\times/K_f^{\times2})_{N\equiv f_0}$ is $G(\Q_p)$-equivalent to an element in
$V(\Z_p)$ for all $p$.  Since $G$ has class number 1, we conclude that
there exists a $G(\Q)$-transformation of $v'$ that lies in $V(\Z)$,
and this is the desired conclusion.  We have proven
Theorem~\ref{fake}, i.e., the average size of the fake 2-Selmer set of 
hyperelliptic curves $C$ of genus g over $\Q$, when ordered by height,
is $o(2^{-g})$.  
As explained in the introduction, by the work of Bruin--Stoll~\cite{BruinStoll} and Skorobogatov \cite{Sk} (see also the related works of Stoll~\cite{Stoll2} and Scharaschkin~\cite{Sch}), this also implies
Corollary~\ref{brauer}.

\begin{remark}{\em 
By the sieve method of \cite{geosieve}, it should be possible to prove that (\ref{finaldensity2}) in fact gives the exact average size of the fake 2-Selmer set over {locally soluble} hyperelliptic curves of genus $g$.  In the case $g=1$, we explain how this can be done in the next section.}
 \end{remark}

\section{Examples}                                                                                          \label{examples}
   
   In this section, we work out some cases of small genus to give an idea of what can be obtained for individual genera. 
    In particular, we show that for all $g\geq 0$, a positive proportion of hyperelliptic curves of genus $g$ over $\Q$ have no rational points, and for $g\geq 1$, a positive proportion fail the Hasse principle due to the Brauer--Manin obstruction. 
   
   \subsection{Genus 0}
                                                                                                            
We start by examining more closely the case of genus 0. In this case, for an an odd prime $p$, we note that a binary quadratic form in $x$ and $y$ over $\Z_p$ that is $\SL_2(\Z_p)$-equivalent to one of the shape $ux^2+p(bx+cy)y$ modulo $p$, where $c\not\equiv 0$ (mod $p$) and $u$ is a nonresidue, does not represent any square in $\Z_p$.  The $p$-adic density of such binary quadratic forms over $\Z_p$ is $(p-1)^2/(2p^2)$. Therefore, the estimate (\ref{finaldensity}) in the case of $g=0$ ($n=2$) may 
be improved to 
                                         
\begin{equation}                                                                                          \label{g0}  
\ll \prod_{p>2}\Bigl(1 - \frac{(p-1)^2}{2p^2}\Bigr).
\end{equation}         
This gives an upper bound on the density of hyperelliptic curves (\ref{hypereq}) of genus 0, when ordered by height, that have a rational point. 
                                                                                                                                    
                                                                                                                                    Since the product (\ref{g0}) converges to 0, it follows that 0\% of genus 0 curves $z^2 = ax^2 + b    
x y + c y^2$, when ordered by $\max \{|a|,|b|,|c|\}$, have a rational point. Since genus~0 curves satisfy the     
Hasse principle, all such genus~0 curves without a rational point fail to have a rational point due     
to not having a point locally at some place.Ê                                                               
                                                                                                            
\subsection{Genus 1}                                                                                        
                                                                                                            
The first interesting and more nontrivial case of Theorem~\ref{main} occurs when $g=1$ (i.e., $n=4$). In this case, the Êrepresentation we have been considering is essentially the same as the one considered in \cite{BS3}, although we are using here the action of a different group, $\SL_4^\pm$, instead of a quotient of a form of $\GL_2 \times \SL_4$.  The locally soluble orbits of the representation of $\SL_2(\Q)\times \SL_4(\Q)$ on $V(\Q)=\Q^2\otimes\Sym_2\Q^4$ parametrize elements of the 4-Selmer groups of elliptic curves. The $\SL_2(\Q)$-equivalence class of the invariant binary quartic form $f_v$ of such an element $v\in V(\Q)$ gives a genus 1 hyperelliptic curve $C_v:z^2=f_v(x,y)$ and corresponds to a 2-Selmer element of $E_v=\Jac(C_v)$, namely, the double of the 4-Selmer element of $E_v$ associated to $v$.   If the elliptic curve $E_v$ is expressed in Weierstrass form as $y^3=x^3+A_vx+B_v$, then $A_v$ and $B_v$ give the two independent invariants of this representation of $\SL_2(\Q)\times \SL_4(\Q)$ on $V(\Q)$.  See \cite{CFS} or \cite{BH} for details on these parametrizations.

When only $\SL_4^\pm(\Q)$ (or $\SL_4(\Q)$) is acting on our representation $V(\Q)$, then the coefficients of the invariant binary quartic form, rather than $A_v$ and $B_v$, generate the ring of polynomial invariants.  Hence~the 2-Selmer element of the elliptic curve $E_v$ corresponding to the locally soluble curve $C_v:z^2=f_v(x,y)$, rather than the Jacobian elliptic curve $E_v$, is used to give a complete set of  invariants for this representation.  This is the essential difference between the orbits of the actions of the two groups on~$V(\Q)$.
 
Now the fake 2-Selmer set of a locally soluble genus one hyperelliptic curve $C$ is naturally in bijection with the set of 4-Selmer elements of $\Jac(C)$, modulo the action of the involution on $C$, whose double yields the 2-Selmer element of $\Jac(C)$ associated to $C$. We have shown that the expression (\ref{finaldensity2}) gives an upper bound on the average size of this fake 2-Selmer set as one ranges over all locally soluble genus one hyperelliptic curves $C$ ordered by height.  Using the same sieve methods discussed in \cite{BS3}, we can show that this upper bound is in fact also a lower bound! 

In the case of genus one, the quantity in (\ref{finaldensity2}) is easily seen to equal 1, because the product over all
places of $$\frac{\#({\Jac(C)(\Q_\nu)}/{2\Jac(C)(\Q_\nu)})}
{\#\Jac(C)[2](\Q_\nu)}$$ is equal to 1 (indeed it is equal to $2^{-g}$ for $\nu=\infty$, \,$2^g$ for $\nu=2$, and $1$ for all other primes; see, e.g., \cite[Lemma~12.3]{BG}).  This leads to:

 \begin{theorem}\label{fake2}
The average size of the fake 2-Selmer set of locally soluble genus one curves $z^2=f(x,y)$ over $\Q$, when ordered by height, is $1$.
\end{theorem}

Note that this fact alone is (just barely) {\it not} sufficient to prove that a positive proportion of the genus one curves (\ref{hypereq}) fail the Hasse principle!  
Indeed, a priori it is possible that 100\% of the 2-Selmer sets of locally soluble genus one curves $C:y^2=f(x)$ have size 1, and that all such curves have a rational point.
This would mean that the 4-Selmer groups of the Jacobians of 100\% of these curves~$C$ contain exactly two order 4 elements (switched by the involution of $C$) which form a homogeneous space for the 2-Selmer group of $\Jac(C)$.  In particular, this would mean that the 2-Selmer groups of 100\% of these locally soluble curves $C$ have size exactly 2.

To show that this leads to a contradiction, we use the following trick from \cite{BS2}.  Namely, we first prove:

\begin{theorem}\label{rootnum}
A positive proportion of locally soluble genus one curves $C:z^2=f(x,y)$ have Jacobian $E$ with root number $+1$. 
\end{theorem}
The analogue of this theorem is proven in \cite{BS2} for Weierstrass elliptic curves $y^2=x^3+Ax+B$ ordered by naive height.  Theorem~\ref{rootnum} can be proven by the same method: one shows exactly as in~\cite{BS2} that for a positive proportion of locally soluble genus one curves 
$C_f:z^2=f(x,y)$,
we have that $\Jac(C_f)$ and $\Jac(C_{-f})$ have opposite root numbers

By the  theorem of Dokchitser and Dokchitser~\cite{DD}, if $\Jac(C)$ has root number +1 for a genus one curve $C$, then the 2-Selmer rank of $\Jac(C)$ must be even and so, in particular, the size of the 2-Selmer group of $\Jac(C)$ cannot be 2.  Therefore, a positive proportion of  genus one curves $C:z^2=f(x,y)$ have size of 2-Selmer set not equal to 1.  By Theorem~\ref{fake2},  it follows that a positive proportion of genus one curves $z^2=f(x,y)$, when ordered by height, fail the Hasse principle. 
 
 \begin{theorem}
 A positive proportion of genus one curves $z^2=f(x,y)$ over $\Q$, when ordered by height, fail the Hasse principle $($i.e., they have points everywhere locally but do not have any global rational points$)$. 
 \end{theorem}                   

\subsection{Genus $\geq2$}                                                                                        

Let us now examine expression (\ref{finaldensity2}) in the case of genus $g\geq 2$. As in the case of genus 1, we note that if we were to replace $C(\Q_\nu)$ by $\Jac(C)(\Q_\nu)$ in the numerators (of the numerators!) in (\ref{finaldensity2}), it would then evaluate to 1.  It follows that (\ref{finaldensity2}) is less than or equal to 1, i.e., the average size of the 2-Selmer set of 
locally soluble
hyperelliptic curves of genus $g$ over $\Q$
is less than or equal to 1 for any genus $\geq 2$.

Now this average value could equal 1 if and only if the archimedean factor in (\ref{finaldensity2}) was equal to $2^{-g}$, the factor at 2 was equal to $2^g$, and all other factors were equal to 1. Thus to show that the average size (\ref{finaldensity2}) of the 2-Selmer set of hyperelliptic curves over genus $g$ is less than 1, it suffices to show that the factor in (\ref{finaldensity2}) at any one place $\nu$ has size less than the corresponding maximum size ($2^{-g}$ for $\nu=\infty$, $2^g$ for $\nu=2$, or 1 otherwise).  This is in fact true already for the infinite place once the genus is at least 2.  Indeed, we may write the factor at infinity in (\ref{finaldensity2}) explicitly as
\begin{equation}\label{2factor}
 \displaystyle\int_{\substack{C=C(f_0,\ldots,f_{n})/\R\\H(C)<1/2}}\!\!\!\!
\frac{\#(C(\R)/2\Jac(C)(\R))}
{\#\Jac(C)[2](\R)}df_0\cdots df_{n} \,=\, 
2^{g}
\sum_{m=0}^{n/2}\frac{\max\{1,2m\}}{2^{m}}
\mu(I(m)).
\end{equation}
Note that for $g=1$, the latter expression is equal to 1, because $\max\{1,2m\}/2^{m}=1$ for $m=0$, 1, and~2, and $\sum_{m=0}^{n/2}\mu(I(m))=1$ by definition.  However, once $g\geq 2$, the scenario $m=3$ occurs with positive probability (i.e., $\mu(I(3))>0$), and this is enough to show that (\ref{2factor}) is strictly less than $2^g$ once $g\geq 2$.  A similar argument shows that the factors at other primes are also less than 1 once $g\geq2$ (and indeed more primes contribute and more significantly as $g$ gets larger).

We conclude:

 \begin{theorem}\label{eachdensity}
For each $g\geq 1$, a positive proportion of curves $z^2=f(x,y)$ of genus $g$ over $\Q$, when ordered by height, fail the Hasse principle $($i.e., they have points everywhere locally but do not have any global rational points$)$ due to the Brauer--Manin obstruction.
 \end{theorem}                   
 It 
 would be interesting to work out explicit positive lower bounds on the proportions occurring in Theorem~\ref{eachdensity} by evaluating (\ref{finaldensity2}) more precisely for various small values of $g$.  A rough estimate shows that the proportion in Theorem~\ref{eachdensity} exceeds 50\% already once $g\geq 2$, and exceeds 99\% once $g\geq 10$.  
 See~\cite[\S10]{BruinStoll} for some remarkable computations that comfortably corroborate these conclusions in the case $g=2$.

\subsection*{Acknowledgments}

I am extremely grateful to Dick Gross for numerous helpful conversations and for his collaboration on~\cite{BG2}, which formed a huge inspiration for this work.  I am also extremely grateful to Wei Ho and Arul~Shankar for their important and related collaborations on \cite{BH} and \cite{BS3}, respectively, which were also inspirational.   Many of the ideas of this paper formed also through very helpful discussions with Nils~Bruin, Jean-Louis Colliot-Th\'el\`ene, Bjorn Poonen, Michael Stoll, Terry Tao, Jerry Wang, and Melanie Wood; it is a pleasure to thank them all. 
 
 The author was supported by a Simons Investigator Grant and NSF grant~DMS-1001828.


\begin{thebibliography}{10}

\bibitem{hclI}
M.\ Bhargava, Higher composition laws I: A new view on Gauss composition, and quadratic \linebreak generalizations, {\it Ann.\ of Math.} {\bf 159} (2004), no.\ 1, 217--250.

\bibitem{hclII}
M.\ Bhargava, Higher composition laws II: On cubic analogues of Gauss composition, {\it Ann.\ of Math.} {\bf 159} (2004), no.\ 2, 865--886.

\bibitem{dodpf}
M.\ Bhargava, The density of discriminants of quintic rings and
fields, {\it Ann.\ of Math.} {\bf 172} \linebreak (2010), 1559--1591.

\bibitem{geosieve}
M.\ Bhargava, The geometric sieve and the density of squarefree values of polynomial discriminants and other
invariant polynomials, in progress.


\bibitem{BG}
M.\ Bhargava and B.\ Gross, Arithmetic invariant theory, 
{\tt http://arxiv.org/abs/1206.4774}.

\bibitem{BG2}
M.\ Bhargava and B.\ Gross, The average size of the 2-Selmer group of Jacobians of hyperelliptic curves having a rational Weierstrass point, 
{\tt http://arxiv.org/abs/1208.1007}.

\bibitem{AITII}
M.\ Bhargava, B.\ Gross, and X.\ Wang, 
Pencils of quadrics and the arithmetic of hyperelliptic curves,
in progress.

\bibitem{BH}
M.\ Bhargava and W.\ Ho, Coregular spaces and genus one curves, 
{\tt http://arxiv.org/abs/1306.4424}.


\bibitem{BS}
M.\ Bhargava and A.\ Shankar, Binary quartic forms having bounded
invariants, and the boundedness of the average rank of elliptic
curves, {\tt http://arxiv.org/abs/1006.1002}.

\bibitem{BS2}
M.\ Bhargava and A.\ Shankar, Ternary cubic forms having bounded
invariants, and the existence of a positive proportion of elliptic curves having
rank 0, {\tt http://arxiv.org/abs/1007.0052}.

\bibitem{BS3}
M.\ Bhargava and A.\ Shankar, The average number of elements in the 4-Selmer groups
of elliptic curves is 7, preprint.

\bibitem{BS4}
M.\ Bhargava and A.\ Shankar, The average size of the 5-Selmer groups
of elliptic curves is 6, and the average rank is less than 1, preprint. 


\bibitem{BM}
B.\ J.\ Birch and J.\ R.\ Merriman, 
Finiteness theorems for binary forms, {\it Proc.\ London Math.\ Soc.} {\bf s3-24} (1972), 385--394.



\bibitem{BoHa}
A.\ Borel and Harish-Chandra, 
Arithmetic subgroups of algebraic groups,
{\it Ann.\ of Math.} {\bf 75} (1962), 485--535. 





\bibitem{BruinStoll}
N.\ Bruin and M.\ Stoll, Two-cover descent on hyperelliptic curves,
{\it Math.\ Comp.} {\bf 78} (2009), 2347--2370.


\bibitem{Car}
M.\ Car, Factorization dans $\F_q[x]$, {\it C.\ R.\ Acad.\ Sci.\ Paris Ser.\ I}
{\bf 294} (1982), 147--150.


\bibitem{Cassels2}
J.\ W.\ S.\ Cassels, The Mordell-Weil group of curves of genus 2, in: {\it Arithmetic and Geometry I}, Birkh\"auser, Boston (1983), 27--60.



\bibitem{CS}
J.-L.\ Colliot-Th\'el\`ene and J.-J.\ Sansuc, La descente sur les vari\'et\'es rationnelles
II, {\it Duke Math.\ J.} {\bf 54}  (1987), 375--492.

\bibitem{CFS} J.\ Cremona, T.\ Fisher, and M.\ Stoll, Minimisation and
  reduction of $2$-, $3$- and $4$-coverings of elliptic curves, 
  {\it Algebra \& Number Theory} {\bf 4} (2010), no.\ 6, 763--820.

\bibitem{Davenport1}
H.\ Davenport, On a principle of Lipshitz,
{\it J.\ London Math.\ Soc.} {\bf 26} (1951), 179--183.  
Corrigendum: ``On a principle of Lipschitz",  {\it J.\ London Math.\
  Soc.} {\bf 39} (1964), 580. 
  
 \bibitem{DPSZ} 
A.\ Dembo, B.\ Poonen, Q.-M.\ Shao, and O.\ Zeitouni, 
  Random Polynomials Having Few or No Real Zeros,
{\it J.\ Amer.\ Math.\ Soc.} {\bf 15} (2002), 857--892.


\bibitem{DR} 
  U.\ V.\ Desale and S.\ Ramanan. Classification of vector bundles of rank 2 on hyperelliptic
curves, {\it Inv.\ Math.} {\bf 38} (1976), 161--185.

\bibitem{DD}
T.\ Dokchitser and V.\ Dokchitser, On the Birch--Swinnerton-Dyer
quotients modulo squares, {\it Ann.~of Math.} {\bf 172} (2010), 567--596.

\bibitem{D}
R.\ Donagi, Group law on the intersection of two quadrics,
{\it Annali della Scuola Normale Superiore di Pisa} {\bf 7} (1980), 217--239.





%


\bibitem{HW}
K.\ Hardy and K.\ S.\ Williams, The class number of pairs of
positive-definite binary
quadratic forms, {\it Acta Arith.} {\bf 52} (1989), no.\ 2, 103-–117.

  

\bibitem{Kn}
A.\ W.\ Knapp, {\it Lie groups beyond an introduction}, Second edition, Progress in Mathematics {\bf 140},
{\it Birkh\"auser Boston, Inc., Boston, MA}, 2002.

\bibitem{KK}
A.\ Knopfmacher and J.\ Knopfmacher,
Counting irreducible factors of polynomials over a finite field,
{\it Discrete Math.} {\bf 112} (1993), 103--118.


%






\bibitem{mc}
W.\ G.\ McCallum, On the method of Coleman and Chabauty, {\it Math.\ Ann.} {\bf 299} (1994), no.\  3, 565--596.


\bibitem{Morales1}
J. Morales, The 
classification of pairs of binary quadratic forms, {\it Acta Arith.}
{\bf 59} (1991), no. 2, 105–121.

\bibitem{Morales}
J.\ Morales, On some invariants for systems of quadratic forms over
the integers,
{\it J.\ Reine Angew.\ Math.} {\bf 426} (1992), 107-–116.

\bibitem{Nakagawa}
J.\ Nakagawa, Binary forms and orders of algebraic number fields, {\it Invent.\ Math.} {\bf 97} (1989), 219--235.



\bibitem{PRAG} V.\ Platonov and A.\ Rapinchuk, Algebraic groups and
  number theory, Translated from the 1991 Russian original by Rachel
  Rowen, {\it Pure and Applied Mathematics}{\bf 139}, Academic Press, Inc.,
  Boston, MA, 1994.
  


\bibitem{PSc}
B.\ Poonen and E.\ F. Schaefer, Explicit descent for Jacobians of cyclic covers of the projective line, {\it J.\ reine angew.\ Math.} {\bf 488} (1997), 141--188.


\bibitem{PS}
B.\ Poonen and M.\ Stoll, {A local-global principle for densities},
  {\it Topics in number theory}, {\it Math.\ Appl.} {\bf 467},
  Kluwer Acad.\ Publ., Dordrecht, 1999, pp.~241--244.

\bibitem{PS2}
B.\ Poonen and M.\ Stoll, 
Most odd degree hyperelliptic curves have only one rational point,
{\tt http://arxiv.org/abs/1302.0061}.


 \bibitem{R}
 M.\ Reid, The complete intersection of two or more quadrics, Ph.D.\ Thesis, Trinity College, 
 Cambridge (1972).


\bibitem{Sch}
V.\ Scharaschkin, {\it Local-global problems and the Brauer-Manin obstruction}, Ph.D.\ thesis, University of Michigan (1999).

\bibitem{Schwarz}
G.\ W.\ Schwarz, Representations of simple {L}ie groups with a free module of covariants, 
{\it Invent. Math.} {\bf 50} (1978), 1--12.




\bibitem{SW}
A.\ Shankar and X.\ Wang, Average size of the 2-Selmer group for monic even hyperelliptic curves, {\tt http://arxiv.org/abs/1307.3531}.


\bibitem{Sk}
A.\ N.\ Skorobogatov, Beyond the Manin obstruction, {\it Invent.\ Math.} {\bf 135} (1999), 399--424.


\bibitem{Stoll}
M.\ Stoll, Implementing $2$-descent for Jacobians of hyperelliptic curves,
{\it Acta Arith.} {\bf 98} (2001), 245--277.


\bibitem{Stoll2}
M.\ Stoll, Finite descent obstructions and rational points on curves, {\it Algebra \& Number Theory} {\bf 1} (2007), no.\ 4, 349Ð391.



 \bibitem{Wth}
 X.\ Wang, {\it Pencils of quadrics and Jacobians of hyperelliptic curves}, Ph.D.\ Thesis, Harvard University, 2013.

\bibitem{W}
X.\ Wang, Maximal linear spaces contained in the the base loci of pencils of quadrics,
{\tt http://arxiv.org/abs/1302.2385}.

\bibitem{Woodth}
M.\ M.\ Wood, {\it Moduli spaces for rings and ideals}, Ph.\ D.\ thesis, Princeton University, 2008. 

\bibitem{Wood}
M.\ M.\ Wood, 
Rings and ideals parametrized by binary $n$-ic forms, {\it J.\ London Math.\ Soc.} (2) {\bf 83} (2011), 208--231. 

\bibitem{Wood1}
M.\ M.\ Wood, Parametrization of ideal classes in rings associated to binary forms, {\it J.\ reine  angew.\ Math.}, to appear.
 
\end{thebibliography}
\end{document}